\documentclass[11pt]{amsart}
\usepackage{amsfonts,amssymb,amsmath,amsthm}
\usepackage{url}
\usepackage{graphicx} 
\usepackage[all]{xy}

\usepackage{enumerate}
\usepackage{diagbox}
\usepackage{shuffle}

\usepackage{amsmath,amsfonts,amsthm,url,color,amssymb}
\usepackage{graphicx}
\usepackage{algpseudocode, algorithm}
\usepackage{hyperref}
\usepackage{todonotes}

\usepackage[norefs,nocites]{refcheck}

\newcommand{\Z}{\mathbb Z}

\newcommand{\gi}{$(G,*)$}

\newtheorem{theorem}{Theorem}[section]
\newtheorem{lemma}[theorem]{Lemma}

\newtheorem{corollary}[theorem]{Corollary}
\newtheorem{proposition}[theorem]{Proposition}
\theoremstyle{definition}
\newtheorem{definition}[theorem]{Definition}

\theoremstyle{remark}
\usepackage{amssymb}
\newtheorem{remark}[theorem]{Remark}
\usepackage{enumerate}

\author{Rafael Bezerra dos Santos}
\address{Departamento de Matem\'{a}tica, Universidade Federal de Minas Gerais, UFMG, Belo Horizonte, MG, 30270-901, Brazil.}
\curraddr{}
\email{rafaelsantos23@mat.ufmg.br}
\thanks{The first author was partially supported by CNPq (Brazil), grant 312058/2025-0, and by FAPEMIG (Brazil), grant RED-00133-21.}
\thanks{}

\author{Lucas Reis}
\address{Departamento de Matem\'{a}tica, Universidade Federal de Minas Gerais, UFMG, Belo Horizonte, MG, 30270-901, Brazil.}
\curraddr{}
\email{lucasreismat@mat.ufmg.br}
\thanks{The second author was supported by CNPq (Brazil), grants 310583/2025-0 and 420721/2025-8, by FAPESP (Brazil), grant 25/21886-4 and by FAPEMIG (Brazil), grant APQ-02427-26. Corresponding author.}

\keywords{superinvolution, polynomial identity, central polynomial, cocharacter, growth}
\subjclass[2010]{16R50, 16W50, 16R10}
\begin{document}

\title[Identities, central polynomials and cocharacters of $(M_2(F), (G,*))$]{Polynomial identities, central polynomials and cocharacters of $M_2(F)$ with $G$-graded involution} 

\begin{abstract} 
Let $F$ be a field of characteristic zero, $G$ be a finite abelian group and $M_2(F)$ be the algebra of $2\times 2$ matrices over $F$. In this paper, we consider all $G$-graded involutions on $M_2(F)$ and determine the generators of the $T_{(G,*)}$-ideal of identities and of the $T_{(G,*)}$-space of central polynomials of $M_2(F)$. Moreover, we explicitly compute the $\langle n\rangle$-cocharacter and the $n$-th
\gi-codimensions for most of the gradings.  For the case of gradings by the Klein group, the  \gi-central codimensions and proper central \gi-codimensions are also explicitly computed.
\end{abstract}

\maketitle

\section{Introduction}
Polynomial identities provide a natural way to study associative algebras through the relations satisfied by all their elements. Let $F$ be a field of characteristic zero, let $A$ be an associative algebra over $F$ and let $F\langle X\rangle$ be the free associative algebra on a countable set $X={x_1,x_2,\ldots}$ of noncommuting variables. A polynomial $f(x_1,\ldots,x_n)\in F\langle X\rangle$ is a polynomial identity of $A$ if
$$
f(a_1,\ldots,a_n)=0
$$
for all $a_1,\ldots,a_n\in A$. The set $Id(A)$ of all polynomial identities of $A$ is a $T$-ideal of $F\langle X\rangle$, and an algebra satisfying a nontrivial polynomial identity is called a PI-algebra. Even for matrix algebras, the problem of describing this $T$-ideal is difficult: for $M_n(F)$, no set of generators for $Id(M_n(F))$ is known when $n\geq 3$.

Besides the identities themselves, one may study quantitative and representation theoretic invariants associated with them. If $P_n$ denotes the space of multilinear polynomials of degree $n$, Regev \cite{Regev} introduced the codimension sequence
$$
c_n(A)=\dim_F\frac{P_n}{P_n\cap Id(A)},\qquad n\geq 1,
$$
and proved that this sequence is exponentially bounded whenever $A$ is a PI-algebra. Since $S_n$ acts on $P_n$ by permuting the variables and $P_n\cap Id(A)$ is invariant under this action, the quotient $P_n/(P_n\cap Id(A))$ is an $S_n$-module. Its character, denoted by $\chi_n(A)$, is the $n$-th cocharacter of $A$.

A related problem concerns central polynomials. A polynomial $f\in F\langle X\rangle$ with no constant term is central for $A$ if all its evaluations on $A$ belong to the center of $A$. We denote by $Id^z(A)$ the set of central polynomials of $A$. Clearly,
$$
Id(A)\subseteq Id^z(A),
$$
and, if $f\in Id^z(A)$, then $[f,x]\in Id(A)$.

The algebra $M_2(F)$ is one of the few matrix algebras for which these objects are well understood in the ordinary setting. Drensky \cite{Drensky1} determined generators for $Id(M_2(F))$. Formanek et al. \cite{Formanek} obtained the decomposition of its $n$-th cocharacter, while Okhitin \cite{Okt} described a generating set for $Id^z(M_2(F))$.

These questions admit natural refinements when the algebra carries an additional structure. For a $G$-graded algebra
$$
A=\bigoplus_{g\in G}A^{(g)},
$$
one considers polynomial identities whose evaluations respect the homogeneous components. Bahturin and Drensky \cite{BD1} studied the $G$-gradings on $M_2(F)$ and determined generators for the corresponding $T_G$-ideals of graded identities. For $G=\mathbb{Z}_2$, Brandão and Koshlukov \cite{Plamen2} described the $T_2$-space of central polynomials of $M_2(F)$, while Di Vicenzo \cite{Onofrio} obtained an explicit decomposition of its $\mathbb{Z}_2$-cocharacter. More recently, in~\cite{Diogo} the authors explored the polynomial identities for $M_2(F)$ for all possible gradings in the case where $F$ is a finite field.

A parallel theory arises when the algebra is endowed with an involution. Recall that an involution $*$ is an antiautomorphism of order at most $2$. Polynomial identities, central polynomials, codimensions and cocharacters can then be defined by requiring evaluations to respect the symmetric and skew elements. For $M_2(F)$ endowed with the transpose or symplectic involution, Levchenko \cite{Lev} determined generators for the corresponding $T^*$-ideals of identities. Drensky and Giambruno \cite{Giam1} described the sequence of $*$-cocharacters, and Brandão and Koshlukov \cite{Plamen2} determined generators for the space of central $*$-polynomials. Analogous problems have been considered for $\mathbb{Z}_2$-graded algebras \cite{Plamen3} and, more recently, for $*$-superalgebras \cite{Ana}. In~\cite{BR26}, we studied $M_2(F)$ with its $\mathbb{Z}_2$-grading and transpose superinvolution, which is the unique admissible superinvolution in that setting.

The aim of the present paper is to treat simultaneously the grading and the involution. More precisely, we consider $M_2(F)$ as a \gi-algebra, that is, $M_2(F)$ is endowed with a $G$-grading and a graded involution. We determine generators for its $T_{(G,*)}$-ideal of identities and for its $T_{(G,*)}$-space of central polynomials. We also study the associated codimensions and cocharacters for all possible gradings. Moreover, for gradings by the Klein group, we explicitly obtain formulas for the central \gi-codimensions and proper central \gi-codimensions.

We conclude the introduction by describing the organization of the paper. In Section 2 we recall the basic definitions and results concerning $G$-graded algebras with involution. In Section 3 we introduce the free associative $(G,*)$-algebra, the $\langle n\rangle$-cocharacters and the auxiliary results used throughout the paper. Section 4 is devoted to elementary gradings of $M_2(F)$, while gradings by the Klein group are treated in Section 5.

\section{$G$-graded algebras with involution}

Let $F$ be a field of characteristic zero, $A$ be an associative algebra over $F$ and $G$ be a group. We say that $A$ is a $G$-graded algebra if $A$ can be written as a direct sum of vector spaces $A=\displaystyle \bigoplus_{g\in G}A^{(g)}$ such that $A^{(g)}A^{(h)}\subseteq A^{(gh)},$
 for all $g,h\in G$. The subspaces $A^{(g)}$ are called homogeneous components of degree $g$ of $A$ and an element $a\in A$ is homogeneous of degree $g$ if $a\in A^{(g)}$.
 
 Given a $G$-graded algebra $A$, the set $supp(G)=\{g\in G: A^{(g)}\neq \{0\}\}$ is called the support of the grading.

 An involution $*$ on an algebra $A$ is an antiautomorphism of order at most 2, i.e. a linear map $*:A\to A$ such that $(ab)^*=b^*a^*$ and $(a^*)^*=a,$ for all $a,b\in A.$ Notice that the identity map is an involution on an algebra $A$ if and only if $A$ is a commutative algebra. An algebra $A$ endowed with an involution $*$ is called a $*$-algebra.

 An involution $*$ on a $G$-graded algebra $A$ is a $G$-graded involution (or simply a graded involution) if $(A^{(g)})^*=A^{(g)}$, for all $g\in G$. In this case, we say that the $G$-graded algebra $A$ is a $(G,*)$-algebra. 

 If $A$ is a \gi-algebra, since char$(F)=0$, each homogeneous component $A^{(g)}$ can be written as $A^{(g)}=(A^{(g)})^+\oplus (A^{(g)})^-$, where $(A^{(g)})^+=\{a\in A^{(g)}:a^*=a\}$ denotes the set of symmetric elements of $A^{(g)}$ and $(A^{(g)})^-=\{a\in A^{(g)}:a^*=-a\}$ denotes the set of skew elements of $A^{(g)}$. As a consequence, a \gi-algebra can be written as $A=\displaystyle\bigoplus_{g\in G}((A^{(g)})^+\oplus(A^{(g)})^-).$

 Let $A$ and $B$ be $G$-graded algebras endowed with graded involutions $\star_1$ and $\star_2$, respectively. We say that $A$ and $B$ are isomorphic, as \gi-algebras, if there exists an isomorphism of algebras $\phi:A\to B$ such that $\phi(A^{(g)})=B^{(g)},$ for all $g\in G,$ and $\phi(a^{\star_1})=\phi(a)^{\star_2}$, for all $a\in A$.

 In this paper, we are interested in studying the graded involutions on the algebra $M_2(F)$, the algebra of $2\times 2$ matrices over $F$. Recall that $M_2(F)=span_F\{e_{11}, e_{22}, e_{12}, e_{21}\}$, where $e_{ij}$ denotes the usual unity matrices.

 Bahturin and Giambruno (\cite[Proposition 4.3]{BG1}) proved that if $M_2(F)$ is a $(G,*)$-algebra, then $supp(G)$ is a commutative subset of $G$. For this reason, from now on, we assume that $G$ is an abelian group.

 We first recall the gradings on $M_2(F)$. A $G$-grading on $M_2(F)$ is called elementary if there exists a pair $(g_1,g_2)\in G^2$ such that $e_{ij}\in M_2(F)^{(g_i^{-1}g_j)}.$

 We can give a structure of $K$-graded algebra on $M_2(F)$, where $K=\langle a,b:a^2=b^2=1, ab=ba \rangle=\{1,a,b,ab\}\cong \mathbb{Z}_2\times \mathbb{Z}_2$ is the Klein group as follows: $M_2(F)^{(1)}=span_F\{e_{11}+e_{22}\}, M_2(F)^{(a)}=span_F\{e_{12}+e_{21}\}, M_2(F)^{(b)}=span_F\{e_{11}-e_{22}\}$ and $M_2(F)^{(ab)}=span_F\{e_{12}-e_{21}\}.$ We call this grading the $(-1)$-grading. 
 
 Bahturin and Zaicev (\cite[Theorem 5.1]{BZ1}) proved that if $F$ is an algebraically closed field, then any $G$-grading on $M_2(F)$, up to isomorphism of $G$-graded algebras, is the elementary grading or the $(-1)$-grading. The algebra $M_2(F)$ admits the following four involutions: 

 \begin{itemize}
     \item[1)] $\begin{pmatrix} a& b\\ c&d\end{pmatrix}^{\gamma_1}=\begin{pmatrix} a& c\\ b&d\end{pmatrix}$, the transpose involution;
     \item[2)] $\begin{pmatrix} a& b\\ c&d\end{pmatrix}^{\gamma_2}=\begin{pmatrix} d& -b\\ -c&a\end{pmatrix}$, the symplectic involution;
     \item[3)] $\begin{pmatrix} a& b\\ c&d\end{pmatrix}^{\gamma_3}=\begin{pmatrix} d& b\\ c&a\end{pmatrix}$, the reflection involution;
     \item[4)] $\begin{pmatrix} a& b\\ c&d\end{pmatrix}^{\gamma_4}=\begin{pmatrix} a& -c\\ -b&d\end{pmatrix}$.
 \end{itemize}

In the next two theorems, we present the graded involutions on the algebra $M_2(F)$ (see \cite[Lemmas 1, 2 and 3]{BZ2}).

\begin{theorem}\label{thm:elementary}
    Let $G$ be an abelian group and $M_2(F)$ be equipped with a nontrivial elementary $G$-grading induced by the pair $(g_1,g_2)\in G^2$ and endowed with a graded involution $*$. If $g_1^2=g_2^2,$ then $*=\gamma_1, \gamma_2$ or $\gamma_3.$ Otherwise, $*=\gamma_2$ or $\gamma_3$.
\end{theorem}

\begin{theorem}\label{-1}
    Let $M_2(F)$ be equipped with the $(-1)$-grading and endowed with a graded involution $*$. Then $*=\gamma_1, \gamma_2, \gamma_3$ or $\gamma_4.$
\end{theorem}

\section{The free associative \gi-algebra and the $\langle n\rangle$-cocharacter}

Let $G$ be a finite abelian group. For all $g \in G,$ consider $(X^{(g)})^* = \{x_{i,g}, x^*_{i,g} ~| ~ i \ge 1\}$ a countable set of variables and define $X =\bigcup\limits_{g\in G}(X^{(g)})^*.$
Let $F \langle X| G, * \rangle$ be the free associative \gi-algebra generated by $X$ over $F,$ whose elements are called $(G,*)$-polynomials. Consider $Y=\underset{g\in G}{\bigcup }Y^{(g) }$ and $Z=\underset{g\in G}{\bigcup }Z^{(g) }$, where $Y^{(g)} = \{y_{i,g}=x_{i,g}+x^*_{i,g}: i \ge 1\}$ is the set of
homogeneous symmetric variables of degree $g$ and $Z^{(g)} = \{z_{i,g}=x_{i,g}-x^*_{i,g} : i \ge 1\}$ is the set of
homogeneous skew variables of degree $g.$ 
Then, $\mathcal{F} := F\langle X ~|~ G, *\rangle = F\langle Y \cup Z\rangle.$ 
For any $g\in G$, define
$$\mathcal{F}_g = \mbox{span}_F \{w_{i_1,g_{j_1}}
\cdots w_{i_t,g_{j_t}}
: g_{j_1}\cdots g_{j_t} = g,~ w_i \in \{y_i, z_i\}\}$$
the space of elements that have homogeneous
degree $g$ and observe that $\mathcal{F} =\bigoplus\limits_{g\in G}\mathcal{F}^{(g)}$ is a $(G, *)$-algebra.

\begin{definition} A $(G,*)$-polynomial $$f=f(y_{1,1},\ldots,y_{i_1,1}, z_{1,1},\ldots,z_{j_1,1}, \ldots, y_{1,g_t},\ldots,y_{i_t,g_t}, z_{1,g_t},\ldots,z_{j_t,g_t}) \in \mathcal{F}$$ is a $(G,*)$-identity of a $(G,*)$-algebra $A,$ and we write $f \equiv 0$ on $A$, if $$f(a^+_{1,1},\ldots,a^+_{i_1,1}, a^-_{1,1},\ldots,a^-_{j_1,1}, \ldots, a^+_{1,g_t},\ldots,a^+_{i_t,g_t}, a^-_{1,g_t},\ldots,a^-_{j_t,g_t})=0$$ for all admissible evaluations, i.e., $a_{r, g}^+\in (A^{(g)})^+$ and $a_{r, g}^-\in (A^{(g)})^-$.
\end{definition}

Let $Id^{(G,*)}(A) \subseteq \mathcal{F}$   be the set of all $(G,*)$-identities of $A.$
Notice that $Id^{(G,*)}(A)$ is an ideal invariant
under all endomorphisms of $\mathcal{F}$ that preserve the grading and commute with the involution. This is the $T_{(G,*)}$-ideal of $A.$ 


If $f_1,\ldots, f_n\in \mathcal{F}$, we denote by $\langle f_1,\ldots, f_n \rangle_{T_{(G,*)}}$ the $T_{(G,*)}$-ideal generated by the polynomials $f_1,\ldots, f_n$.

Since $F$ is a field of characteristic zero, $Id^{(G,*)}(A)$ is determined by 
multilinear $(G,*)$-polynomials. Thus, we consider
$$P_n^{(G,*)}=\mbox{span}_F\{w_{\sigma(1)}w_{\sigma(2)}\cdots w_{\sigma(n)}: \sigma \in S_n,~ w_i \in \{y_{i,g}, z_{i,g}\}, ~1 \le i \le n, ~g \in G\}$$ 
the space of multilinear $(G,*)$-polynomials of degree $n.$ 

\begin{definition}
    For $n\ge 1,$ the $n$-th $(G,*)$-codimension of a $(G,*)$-algebra $A$ is defined as $$c_n^{(G,*)}(A):=\dim_F\frac{P_n^{(G,*)}}{P_n^{(G,*)}\cap Id^{(G,*)}(A)}. $$ 
\end{definition}

Observe that $c_n(A)\leq c_n^{(G,*)}(A)\leq 2^n|G|^nc_n(A).$ Thus, by \cite{Regev}, if $A$ is a $(G,\ast)$-algebra satisfying a nontrivial ordinary polynomial identity, then its sequence of $(G, *)$-codimensions is exponentially bounded. 
 For readers interested in studying the asymptotic behavior of such a sequence, we recommend the references \cite{Mara, OSV, Lorena}. 

 Next, we will deal with central polynomials. Given a \gi-algebra $A$, a \gi-polynomial $f\in \mathcal{F}$ is a central \gi-polynomial for $A$ if it has no constant term and any admissible evaluation on $f$ at elements of $A$ belongs to the center of $A$. It is clear that a \gi-polynomial identity of $A$ is a central \gi-polynomial of $A$. The set $$Id^{(G,*),z}(A)=\{f\in \mathcal{F}:\mbox{ $f$ is a central $(G,*)$-polynomial for $A$}\}$$ is not, in general, a $T_{(G,*)}$-ideal of $\mathcal{F}$. However, it is a $T_{(G,*)}$-subspace of $\mathcal{F}$, i.e. a subspace invariant under all graded endomorphisms of $\mathcal{F}$ commuting with the involution $*$. 

We may also define two other sequences attached to a \gi-algebra.

\begin{definition}\label{codimcentral}
    For $n\ge 1,$ the $n$-th central $(G,*)$-codimension of a $(G,*)$-algebra $A$ is defined as $$c_n^{(G,*),z}(A):=\dim_F\frac{P_n^{(G,*)}}{P_n^{(G,*)}\cap Id^{(G,*),z}(A)}. $$ 
\end{definition}

\begin{definition}\label{codimproper}
    For $n\ge 1,$ the $n$-th proper central $(G,*)$-codimension of a $(G,*)$-algebra $A$ is defined as $$c_n^{(G,*),\delta}(A):=\dim_F\frac{P_n^{(G,*)}\cap Id^{(G,*),z}(A)}{P_n^{(G,*)}\cap Id^{(G,*)}(A)}. $$ 
\end{definition}

It is clear that, for all $n\geq 1, c_n^{(G,*)}(A)=c_n^{(G,*),z}(A)+c_n^{(G,*),\delta}(A).$ As a consequence, if a \gi-algebra $A$ satisfies a nontrivial ordinary identity, then the sequences $\{c_n^{(G,*),z}(A)\}_{n\geq 1}$ and $\{c_n^{(G,*),\delta}(A)\}_{n\geq 1}$ are exponentially bounded.

 If $f_1,\ldots, f_n\in \mathcal{F}$, we denote by $\langle f_1,\ldots, f_n \rangle^{T_{(G,*)}}$ the $T_{(G,*)}$-space generated by the polynomials $f_1,\ldots, f_n$.


     \begin{remark}\label{rem:center}
         For $n\geq 1$, the center of $M_n(F)$ is isomorphic to $F$. More precisely, $Z(M_n(F))=\{\alpha I_n:\alpha \in F\},$ where $I_n$ denotes the $n\times n$ identity matrix.
     \end{remark}

From now on, we assume that $G=\{g_1=1,g_2,\ldots,g_k\}$ is a finite abelian group of order $k$ and $A$ is a $(G, *)$-algebra. For an integer $n \in \mathbb{N},$ we write $n=n_1+n_2+\cdots+n_{2k},$ where each $n_i$ is a non-negative integer, for $1\le i \le 2k$ and denote by $\langle n \rangle =(n_1, n_2, \ldots, n_{2k})$ a composition of $n$ into $2k$ parts. A multipartition $\langle \lambda \rangle=(\lambda_1, \lambda_2, \ldots, \lambda_{2k})\vdash \langle n \rangle$ means that $\lambda_i\vdash n_i$ for $1 \le i\le 2k$. We write $\langle \lambda \rangle \vdash n$ if this holds for some composition $\langle n \rangle$ of $n. $ 

Let $P_{\langle n \rangle}$ be the space of multilinear $(G,*)$-polynomials  where the first $n_1$ variables are symmetric in homogeneous degree $1,$ the next $n_2$ variables are skew of homogeneous degree $1,$ and so on so that the penultimate $n_{2k-1}$ variables are symmetric of homogeneous degree $g_k$ and the last $n_{2k}$ variables are skew of homogeneous degree $g_k.$ 

Note that there are
$\displaystyle\binom{n}{\langle n \rangle}:=\displaystyle\binom{n}{ n_1, \ldots, n_{2k} }$ subspaces isomorphic to $P_{\langle n \rangle}$ in $P_n^{(G,*)}$. In fact, we have
\begin{equation} \label{pn-}
P_n^{(G,*)} \cong \displaystyle \bigoplus_{\langle n \rangle } \displaystyle\binom{n}{\langle n \rangle} P_{\langle n \rangle} .
\end{equation}

We consider the quotient spaces $P_{\langle n \rangle}(A) = \dfrac{P_{\langle n \rangle}}{P_{\langle n \rangle}\cap Id^{(G,*)}(A)}, P_{\langle n \rangle}^z(A) = \dfrac{P_{\langle n \rangle}}{P_{\langle n \rangle}\cap Id^{(G,*),z}(A)} \mbox{ and } P_{\langle n \rangle}^{\delta}(A) = \dfrac{P_{\langle n \rangle}\cap Id^{(G,*),z}(A)}{P_{\langle n \rangle}\cap Id^{(G,*)}(A)}$ and define $c_{\langle n \rangle}(A)=\dim_F( P_{\langle n \rangle}(A))$ as the $\langle n \rangle $-codimension of $A$, $c_{\langle n \rangle}^{z}(A)=\dim_F(P_{\langle n \rangle}^z(A))$ as the central $\langle n\rangle$-codimension of $A$ and $c_{\langle n \rangle}^{\delta}(A)=\dim_F(P_{\langle n \rangle}^{\delta}(A))$ as the proper central $\langle n\rangle$-codimension of $A$.

By (\ref{pn-}), the relation between the $n$-th $(G,*)$-codimension of $A$ and its $\langle n \rangle$-codimension is given by
\begin{equation} \label{293-}
		c_n^{(G,*)}(A)= \underset{\langle n \rangle }{\sum} \displaystyle\binom{n}{\langle n \rangle } c_{\langle n \rangle}(A). 
	\end{equation}

\begin{remark}\label{rmk:codimcentral}
    In case of central and proper central $\langle n\rangle$-codimensions, we have that $c_n^{(G,*),z}(A)= \displaystyle\underset{\langle n \rangle }{\sum} \displaystyle\binom{n}{\langle n \rangle } c_{\langle n \rangle}^z(A)$ and $c_n^{(G,*),\delta}(A)= \displaystyle\underset{\langle n \rangle }{\sum} \displaystyle\binom{n}{\langle n \rangle } c_{\langle n \rangle}^{\delta}(A)$.
\end{remark}

There is a natural left action of the group $S_{\langle n \rangle}:= S_{n_1}\times\cdots \times S_{n_{2k}}$ on $P_{\langle n \rangle},$ where each $S_{n_i}$ permutes
the corresponding variables associated with $n_i,$ $1\le i \le 2k.$ Since $P_{\langle n \rangle} \cap Id^{(G,*)}(A)$ is invariant under this
action, $P_{\langle n \rangle}(A)$ inherits a structure of $S_{\langle n \rangle}$-module. It is known that the irreducible $S_{\langle n \rangle}$-characters are
outer tensor products of irreducible $S_{n_i}$-characters which are in one-to-one correspondence with partitions
$\lambda_i \vdash n_i.$ Hence, we consider $\chi_{\langle \lambda \rangle}=\chi_{\lambda_1} \otimes\cdots \otimes \chi_{\lambda_{2k}}$ the irreducible $S_{\langle n \rangle}$-character associated to a multipartition
$\langle \lambda \rangle := (\lambda_1, \ldots, \lambda_{2k}) \vdash \langle n \rangle,$ where $\chi_{\lambda_i}$ is the irreducible $S_{n_i}$-character associated to $\lambda_i.$ Moreover, its degree
is given by $d_{\langle \lambda \rangle}=d_{\lambda_1} \cdots d_{\lambda_{2k}},$ where $d_{\lambda_i}$ is the degree of $\chi_{\lambda_i}$. By complete reducibility we may consider

\begin{equation}\label{cocharacter}\chi_{\langle n \rangle}(A)=\sum\limits_{\langle \lambda \rangle \vdash \langle n \rangle} m_{\langle \lambda \rangle} \chi_{\langle \lambda \rangle},\end{equation}
the decomposition of the $\langle n \rangle$-character of the space $P_{\langle n \rangle}(A)$ into irreducible characters, called $\langle n \rangle$-cocharacter of $A,$
where $m_{\langle \lambda \rangle}$ is the multiplicity of $\chi_{\langle \lambda \rangle}$. 


By (\ref{cocharacter}) we notice that
 \begin{equation} \label{cmultin-}
     c_{\langle n \rangle }(A)= \chi_{\langle n \rangle}(A)(1)=\sum\limits_{\langle \lambda \rangle \vdash \langle n \rangle} m_{\langle \lambda \rangle} d_{\langle \lambda \rangle}.
 \end{equation}

 To obtain more precise information about the multiplicities $m_{\langle \lambda \rangle}$ that appear in the decomposition of the $\langle n \rangle$-cocharacter $\chi_{\langle n \rangle}(A)$ described in (\ref{cocharacter}), we use the representation theory of the general linear group $GL_m$ in terms of $(G, *)$-algebras. The
details can be found in [\cite[Section 12.4]{livro} and in \cite{OSV}.

For $m \ge 1,$ define $X^m=
\bigcup\limits_{g\in G}(Y^{(g)})^m\cup (Z^{(g)})^m$, where for $g\in G$, we consider
$$(Y^{(g)})^m=\{y_{1,g}, \ldots, y_{m,g}\} \;\;\mbox{and}\;\;(Z^{(g)})^m=\{z_{1,g}, \ldots, z_{m,g}\}.$$ 

Denote by $F_m:=F\langle X^m|G,*\rangle$ the free associative $(G,*)$-algebra generated by $X^m$ over $F.$ Define $F_m^n$ the subspace of homogeneous polynomials in $F_m$ with degree $n \ge m$ and notice that the group $GL_m^{2k}:= GL_m\times \cdots \times GL_m,$ the direct product of $2k$-copies of $GL_m,$ acts diagonally on $F_m^n.$ Since $F_m^n\cap Id^{(G,*)}(A)$ is invariant under this action, we have that the space

 $$F_m^n(A)=\frac{F_m^n}{F_m^n\cap Id^{(G,*)}(A)}$$ has a structure of $GL_m^{2k}$-module. Hence, we can consider $\psi_n^{(G,*)}(A)$ its $GL_m^{2k}$-character, called $n$-th $GL_m^{2k}$-cocharacter of $A.$ It is known (see \cite{livro})  that there exists a one-to-one correspondence between irreducible $GL_m^{2k}$-modules and multipartitions $\lambda=(\lambda_1, \ldots, \lambda_{2k})\vdash \langle n \rangle,$ where $\lambda_i$ is a  partition of $n_i$ with at most $m$ parts, for $1 \le i \le 2k.$ Since char$(F)=0,$ by complete reducibility, we may write 
 
 \begin{equation} \label{eq2coca-}
\psi_n^{(G,*)}(A)=  \displaystyle \sum_{\langle n\rangle  } \underset{\scriptsize{\begin{array}{c}
{\langle \lambda\rangle}\vdash {\langle n\rangle} \\
			h( \lambda )\leq m
	\end{array}}}{\sum} \widetilde{m}_{{\lambda}}\psi_{{\lambda}},
\end{equation}  
where $\psi_{\langle \lambda \rangle}$ is the irreducible $GL_m^{2k}$-character associated to
the multipartition $\langle \lambda \rangle$ and $h(\langle \lambda \rangle)$ is the maximum value of the heights $h(\lambda_i),$ $1 \le i \le 2k,$ of the Young diagrams corresponding to the partitions $\lambda_i\vdash n_i.$ 

\begin{theorem}\label{multip-}
If $\chi_{\langle n \rangle}(A)$ and $\psi_n^{(G, *)}(A)$ are the $\langle n \rangle$-cocharacter and the $GL_m^{2k}$-cocharacter of $A$ as given in (\ref{cocharacter}) and (\ref{eq2coca-}), respectively, then $m_{\langle \lambda \rangle}=\tilde{m}_{\langle \lambda \rangle}$ for all multipartitions $\langle \lambda\rangle\vdash \langle n \rangle$ such that $h(\langle \lambda \rangle)\le m.$
	
\end{theorem}


By [\cite{livro}, Theorem 12.4.12], each irreducible $GL^{2k}_m$-module is generated by a non-zero polynomial $f_{\langle \lambda \rangle}$  called the
highest weight vector associated to the multipartition $\langle \lambda \rangle$ and it is given by $$f_{{\langle \lambda \rangle}} =\prod\limits_{j=1}^{(\lambda_1)_1}St_{{h_j}(\lambda_1)}(y_{1,1}, \ldots, y_{{h_j}(\lambda_1),1})  \cdots \prod\limits_{j=1}^{(\lambda_{2k})_1}St_{{h_j}(\lambda_{2k})}(z_{1,g_k}, \ldots, z_{{h_j}(\lambda_{2k}),g_k})$$

	

\noindent where $St_r(x_1, \ldots , x_r) = \sum\limits_{\sigma \in S_r}\mbox{sgn}(\sigma)x_{\sigma(1)}\cdots x_{\sigma(r)}$ is the standard polynomial of degree $r$ and $h_j(\lambda_{i})$  represents the height of the $j$-th column of the Young diagram $T_{\lambda_{i}}$ associated to the partition $\lambda_i\vdash n_i.$ It is known that
every polynomial $f_{\langle \lambda \rangle}$ is linearly generated by the polynomials $f_{T_{\langle \lambda \rangle}}$ as we will see below.

For a multipartition $\langle \lambda \rangle =
(\lambda_1, \ldots, \lambda_{2k}) \vdash \langle n \rangle$ we consider the multitableau $T_{\langle \lambda \rangle} = (T_{\lambda_1} , ~\ldots ,~ T_{\lambda_{2k}} )$ formed by $2k$ Young tableaux, which
is filled by placing the numbers from $1$ to $n$ in ascending order from top to bottom. We define the standard
multitableau to be the one such that the integers $1, ~\ldots ,~ n$ in this order, fill in from top to bottom, column by
column, the tableau $T_{\lambda_{1}}$ to the tableau $T_{\lambda_{2k}} .$

Consider $\sigma \in S_n$ the only permutation that changes the standard
multitableau to the multitableau $T_{\langle \lambda \rangle}.$ The highest weight vector $f_{T_{\langle \lambda \rangle}}$ corresponding to the multitableau
$T_{\langle \lambda \rangle}$ is defined as
$f_{T_{\langle \lambda \rangle}}:= f_{{\langle \lambda \rangle}}\sigma^{-1},$  
where the right action of $S_n$ on $F^n_m(A)$ is defined by exchanging the places of the variables in each monomial.   

The next result relates the highest weight vectors  to the multiplicities in Theorem \ref{multip-}. 

\begin{theorem}\label{multiplicity-}
	The multiplicity $\tilde{m}_{\langle \lambda \rangle}$ in (\ref{eq2coca-}) is non-zero if, and only if, there exists a multitableau $T_{\langle \lambda \rangle},$ such that $f_{T_{\langle \lambda \rangle}} \notin Id^{(G,*)}(A).$ 
	Moreover, $\tilde{m}_{\langle \lambda \rangle}$ is equal to the maximum number of highest weight vectors  associated to the multitableaux of type $\langle \lambda \rangle$ that are linearly independent in $F_{m}^n(A).$ 
	
\end{theorem}

\begin{remark}\label{altura}
    Let $A$ be a \gi-algebra and write $A=\displaystyle\bigoplus_{g\in G} (A^{(g)})^+\oplus (A^{(g)})^-$, with $\dim_F((A^{(g_i)})^{\epsilon})=d_i^{\epsilon}, g_i\in G=\{g_1=1,g_2,\ldots,g_k\},\epsilon\in\{+,-\}$. For a multipartition $\langle \lambda \rangle =
(\lambda_1, \ldots, \lambda_{2k}) \vdash \langle n \rangle$, consider the multitableau $T_{\langle \lambda \rangle} = (T_{\lambda_1} , ~\ldots ,~ T_{\lambda_{2k}} )$ and let $f_{T_{\langle \lambda \rangle}}$ be the highest weight vector associated to $T_{\langle \lambda \rangle}$. If, for some $1\le i\le k$,  $h_1(\lambda_{2i-1})>d_i^+$ or $h_1(\lambda_{2i})>d_i^-$, then $f_{T_{\langle \lambda \rangle}} \in Id^{(G,*)}(A).$ As a consequence, $\tilde{m}_{\langle \lambda \rangle}= 0$ in (\ref{eq2coca-}).
\end{remark}

	 In the next sections, we determine the sets $Id^{(G,*)}(A)$ and $Id^{(G,*),z}(A)$, the explicit decomposition of $\chi_{\langle n\rangle}(A)$ and we obtain asymptotics for certain codimensions, when $A$ is the algebra $M_2(F)$ endowed with a $G$-graded involution $*$.

We end this section with some relevant auxiliary lemmas. In the next results, $F$ is a field of arbitrary characteristic.

\begin{lemma}\label{lem:vanish} Let $F$ be an infinite field and let $f$ be a polynomial over $F$ in $n$ commuting variables. If $f$ vanishes entirely on $F^n$, then $f$ equals the zero polynomial. \end{lemma} 

Given $\mathcal{F}=F \langle X|G,* \rangle$, we denote by $\mathcal{F}^C=F[X|G,*]$ the free associative and commutative \gi-algebra on $X$ over $F$. 

\begin{definition} Let $f\in \mathcal{F}$. We denote by $f^C\in \mathcal{F}^C$ the polynomial induced by $f$ through the relations $x_ix_j=x_jx_i$ for each $x_k\in\{y_{k,g},z_{k,g}\}$ with $g\in G$. \end{definition}

For instance, if $$f(y_{1,g_1}, z_{1,g_2})=y_{1,g_1}z_{1,g_2}+z_{1,g_2}y_{1,g_1}\in \mathcal{F},$$ then $f^C(y_{1,g_1}, z_{1,g_2})=2y_{1,g_1}z_{1,g_2}\in \mathcal{F}^C$. The polynomial $f^C$ can be zero even if $f\in \mathcal F$ is nonzero: $f(y_{1, g_1}, z_{1, g_2}, z_{1, g_3})=y_{1,g_1}z_{1,g_2}z_{1, g_3}-z_{1,g_2}z_{1, g_3}y_{1,g_1}\in \mathcal F$ satisfies $f(y_{1, g_1}, z_{1, g_2}, z_{1, g_3})^C=0\in \mathcal F^C$. The following lemma, crucial in the proofs of our main results, provides a sufficient condition under which the vanishing of $f^C$ implies the vanishing of $f$.

\begin{lemma}\label{lem:com} Let $F$ be a field and let $f\in \mathcal{F}$ be a polynomial not admitting two distinct monomials $M_1, M_2$ with $(M_1)^C=(M_2)^C$. If $f^C\in \mathcal{F}^C$ is the zero polynomial, the same holds for $f$. In particular, if $F$ is infinite, $f$ contains $n$ distinct variables and $f^C$ vanishes entirely on $F^n$, then $f$ is the zero polynomial. \end{lemma} 

\begin{proof} By hypothesis, distinct monomials in $f$ induce distinct monomials in $\mathcal F^C$, so $f^C=0$ implies $f=0$. The second statement follows directly from the first one and Lemma~\ref{lem:vanish}. \end{proof}

The next lemma is fundamental for determining the $T_{(G,*)}$-ideal of $M_2(F)$ when endowed with the involutions $\gamma_i$.

\begin{lemma}\label{lem:id-equiv}    Let $F$ be a field of characteristic different from $2$, let $G$ be a finite abelian group and let $A$ be an algebra. Suppose that $(A_1,G,*)$ and $(A_2,G,\star)$ are two $G$-graded involutions on $A$ such that
\[
\{(A_1^{(g)})^+,(A_1^{(g)})^-\}=\{(A_2^{(g)})^+,(A_2^{(g)})^-\}
\]
for every $g\in G$. Define $\varphi:F\langle X\mid G,*\rangle\to F\langle X\mid G,\star\rangle$ by
\[
\bigl(\varphi(y_{i,g}),\varphi(z_{i,g})\bigr)=
\begin{cases}
(\bar y_{i,g},\bar z_{i,g}),&\text{if }(A_1^{(g)})^+=(A_2^{(g)})^+,\\
(\bar z_{i,g},\bar y_{i,g}),&\text{otherwise},
\end{cases}
\]
and extend it multiplicatively and linearly. Then $\varphi$ is a bijection satisfying the following. For items (b) and (d), assume in addition that the map $g\mapsto\varepsilon_g\in\mathbb Z_2$, defined by $\varepsilon_g=0$ when $(A_1^{(g)})^+=(A_2^{(g)})^+$ and $\varepsilon_g=1$ otherwise, is a group homomorphism.

\begin{enumerate}[(a)]
   
    \item $\varphi(Id^{(G,*)}(A_1))= Id^{(G,\star)}(A_2)$;
    \item if $Id^{(G,*)}(A_1)=\langle f_1,\ldots,f_k\rangle_{T_{(G,*)}}$, then
\[
Id^{(G,\star)}(A_2)=\langle\varphi(f_1),\ldots,\varphi(f_k)\rangle_{T_{(G,\star)}};
\]
\item $\varphi(Id^{(G,*),z}(A_1))= Id^{(G,\star),z}(A_2)$;
\item if $Id^{(G,*),z}(A_1)=\langle h_1, \ldots, h_k\rangle^{T_{(G, *)}}$, then $$Id^{(G,\star),z}(A_2)=\langle \varphi(h_1), \ldots, \varphi(h_k)\rangle^{T_{(G, \star)}}.$$

\end{enumerate}

\end{lemma}

    \begin{proof}
For each \(g\in G\), set

$$
\varepsilon_g=
\begin{cases}
0,&\text{if }(A_1^{(g)})^+=(A_2^{(g)})^+,\\
1,&\text{otherwise}.
\end{cases}
$$

By hypothesis, if \(\varepsilon_g=1\), then $(A_1^{(g)})^+=(A_2^{(g)})^-$ and $(A_1^{(g)})^-=(A_2^{(g)})^+.
$

Thus \(\varphi\) is determined on the generators by

$$
\varphi(y_{i,g})=
\begin{cases}
\bar y_{i,g},&\varepsilon_g=0,\\
\bar z_{i,g},&\varepsilon_g=1,
\end{cases}
\qquad
\varphi(z_{i,g})=
\begin{cases}
\bar z_{i,g},&\varepsilon_g=0,\\
\bar y_{i,g},&\varepsilon_g=1.
\end{cases}
$$

Define $
\psi:F\langle X\mid G,\star\rangle\to
F\langle X\mid G,*\rangle
$ on the generators by

$$
\psi(\bar y_{i,g})=
\begin{cases}
y_{i,g},&\varepsilon_g=0,\\
z_{i,g},&\varepsilon_g=1,
\end{cases}
\qquad
\psi(\bar z_{i,g})=
\begin{cases}
z_{i,g},&\varepsilon_g=0,\\
y_{i,g},&\varepsilon_g=1,
\end{cases}
$$

and extend it multiplicatively and linearly. Then $
\psi\circ\varphi=\operatorname{id}_{F\langle X\mid G,*\rangle}$ and $
\varphi\circ\psi=\operatorname{id}_{F\langle X\mid G,\star\rangle}.
$ Hence \(\varphi\) is a bijection. We proceed to the proof of items (a)-(d).

\begin{enumerate}[(a)]
    \item Let \(f\in Id^{(G,*)}(A_1)\) and let $
\eta:F\langle X\mid G,\star\rangle\to A
$ be an arbitrary admissible \((G,\star)\)-evaluation. Define $
\widetilde{\eta}
=
\eta\circ\varphi:
F\langle X\mid G,*\rangle\to A.
$ We claim that \(\widetilde{\eta}\) is an admissible \((G,*)\)-evaluation. Indeed, fix \(g\in G\). If \(\varepsilon_g=0\), then
$
\widetilde{\eta}(y_{i,g})
=
\eta(\bar y_{i,g})
\in (A_2^{(g)})^+
=
(A_1^{(g)})^+,
$
and
$\widetilde{\eta}(z_{i,g})
=\eta(\bar z_{i,g})\in (A_2^{(g)})^-=
(A_1^{(g)})^-$. If $\varepsilon_g=1$, then
$\widetilde{\eta}(y_{i,g})=\eta(\bar z_{i,g})\in (A_2^{(g)})^-=
(A_1^{(g)})^+,$
and
$
\widetilde{\eta}(z_{i,g})
=
\eta(\bar y_{i,g})
\in (A_2^{(g)})^+
=
(A_1^{(g)})^-.
$
Therefore \(\widetilde{\eta}\) is admissible. Since \(f\in Id^{(G,*)}(A_1)\), we obtain $$
0=\widetilde{\eta}(f)
=(\eta\circ\varphi)(f)
=\eta(\varphi(f)).
$$
As \(\eta\) is arbitrary, we have $\varphi(f)\in Id^{(G,\star)}(A_2)
$ and then
$\varphi\bigl(Id^{(G,*)}(A_1)\bigr)
\subseteq
Id^{(G,\star)}(A_2)
$. The other inclusion follows similarly. 


\item 
Suppose now that $
Id^{(G,*)}(A_1)
=
\langle f_1,\ldots,f_k\rangle_{T_{(G,*)}}
$. Observe that \(\varphi\) is an isomorphism of the corresponding free graded algebras which sends variables of each prescribed symmetry and homogeneous degree to variables of the corresponding admissible type for \((G,\star)\). From hypothesis, the map $\varepsilon:G\to\mathbb Z_2$ is a group homomorphism, hence for every homogeneous polynomial $f$ of degree $g$ we have

$$
\varphi(f^*)=(-1)^{\varepsilon_g}\varphi(f)^\star.
$$
Indeed, for a monomial $M=w_{1,g_1}\cdots w_{r,g_r}$ of degree $g=g_1\cdots g_r$,
\[
\begin{aligned}
\varphi(M^*)
=(-1)^{\varepsilon_{g_1}+\cdots+\varepsilon_{g_r}}\varphi(M)^\star\
=(-1)^{\varepsilon_g}\varphi(M)^\star,
\end{aligned}
\]
where the last equality follows precisely from the fact that $\varepsilon$ is a group homomorphism.

Thus conjugation by $\varphi$ gives a bijection between the $(G,*)$-endomorphisms of $F\langle X\mid G,*\rangle$ and the $(G,\star)$-endomorphisms of $F\langle X\mid G,\star\rangle$. Indeed, if $\alpha$ is a $(G,*)$-endomorphism, then
$$
\beta=\varphi\circ\alpha\circ\varphi^{-1}:F\langle X| G,\star\rangle\to 
F\langle X| G,\star\rangle$$
preserves the grading and, for every homogeneous $f\in F\langle X\mid G, \star\rangle$ of degree $g$,
$$
\begin{aligned}
\beta(f^\star)
&=(-1)^{\varepsilon_g}\varphi\bigl(\alpha(\varphi^{-1}(f))^*\bigr)\\
&=(-1)^{2\varepsilon_g}\varphi\bigl(\alpha(\varphi^{-1}(f))\bigr)^\star\\
&=\beta(f)^\star.
\end{aligned}
$$
The converse follows by applying the same argument to $\varphi^{-1}$. Consequently,
$$
\varphi\left(
\langle f_1,\ldots,f_k\rangle_{T_{(G,*)}}
\right)
=
\langle
\varphi(f_1),\ldots,\varphi(f_k)
\rangle_{T_{(G,\star)}}.
$$
From the equality proved in item (a), we conclude that
$$
Id^{(G,\star)}(A_2)
=
\langle
\varphi(f_1),\ldots,\varphi(f_k)
\rangle_{T_{(G,\star)}}.
$$

\item Fix $f\in Id^{(G,*),z}(A_1)$. In particular, for every $P\in F\langle X\mid G,*\rangle$, the polynomial $[f, P]$ is an identity for $(A_1, G, *)$. From item (a), $\varphi([f, P])$ is an identity for $(A_2, G, \star)$. It is clear that $\varphi([f, P])=[\varphi(f), \varphi(P)]$. Since $\varphi$ is a bijection, $\varphi(P)$ represents an arbitrary element of $F\langle X\mid G,\star\rangle$. Thus $\varphi(f)\in Id^{(G,\star),z}(A_2)$. The other inclusion follows in the same way.

\item The proof follows similarly to the one in item (b). 
\end{enumerate}



\end{proof}

The following result is a direct application of Lemma~\ref{lem:id-equiv}.

\begin{corollary}\label{cor:equal}
    Let $F$ be a field of characteristic zero and let $(A_1, G, *)$ and $(A_2, G, \star)$ be as in Lemma~\ref{lem:id-equiv}. Then for every integer $n\ge 1$ we have $c_n^{(G, *)}(A_1)=c_n^{(G, \star)}(A_2), c_n^{(G, *), z}(A_1)=c_n^{(G, \star), z}(A_2)$ and $c_n^{(G, *), \delta}(A_1)=c_n^{(G, \star), \delta}(A_2)$. 
\end{corollary}

\section{$M_2(F)$ with elementary $G$-graded involution}

Let $G$ be a finite abelian group and consider $M_2(F)$ equipped with a nontrivial elementary $G$-grading induced by the pair ${\bf g}=(g_1,g_2)\in G^2$ and endowed with a graded involution $\gamma_i$ as in Theorem \ref{thm:elementary}. In this case, we write $\mathbb{M}_{{\bf g}}^{\gamma_i}.$ First, we will deal with the case that in ${\bf g}=(g_1,g_2)$, we have $g_1^2=g_2^2$. Since $G$ is an abelian group, this is equivalent to $g_1^{-1}g_2=g_2^{-1}g_1$. Let $g=g_1^{-1}g_2.$ Then $g^2=1$, $M_2(F)^{(1)}=Fe_{11}+Fe_{22}, M_2(F)^{(g)}=Fe_{12}+Fe_{21}$ and $ M_2(F)^{(h)}=\{0\}$ if $h\neq 1,g,$ and this algebra is endowed with the involutions $\gamma_1,\gamma_2$ and $\gamma_3$.

This case was recently treated by Cruz and Vieira in \cite{Ana} in the context of $*$-superalgebras, that is, $\mathbb{Z}_2$-graded algebras with graded involutions. In the next theorem, $x_{1,h}\in \{y_{1,h},z_{1,h}\}$ with $h\neq 1,g.$

\begin{theorem}\cite[Theorems 4.9 and 5.8 and Corollary 6.3]{Ana}
    Let $F$ be an infinite field of characteristic different from $2$ and consider $M_2(F)$ equipped with a nontrivial elementary $G$-grading induced by the pair ${\bf g}=(g_1,g_2)\in G^2,$ with $g_1^2=g_2^2$ and endowed with the involutions $\gamma_1,\gamma_2$ and $\gamma_3.$ Then:
    \begin{enumerate}
        \item $Id^{(G,*)}(\mathbb{M}_{{\bf g}}^{\gamma_1})=\langle z_{1,1}, x_{1,h}\rangle_{T_{(G,*)}};$
        \item $Id^{(G,*)}(\mathbb{M}_{{\bf g}}^{\gamma_2})=\langle y_{1,g},[y_{1,1},y_{2,1}], [y_{1,1},z_{1,1}], [z_{1,1},z_{2,1}], x_{1,h}\rangle_{T_{(G,*)}};$
        \item $Id^{(G,*)}(\mathbb{M}_{{\bf g}}^{\gamma_3})=\langle z_{1,g},[y_{1,1},y_{2,1}], [y_{1,1},z_{1,1}], [z_{1,1},z_{2,1}], x_{1,h}\rangle_{T_{(G,*)}};$
    \end{enumerate}
\end{theorem}

In the same paper, the authors described the generators of the $T_{(G,*)}$-space of central polynomials of $M_2(F)$.

\begin{theorem}\cite[Theorems 4.11 and 5.11 and Corollary 6.4]{Ana} Let $F$ be an infinite field of characteristic different from 2 and consider $M_2(F)$ equipped with a nontrivial elementary $G$-grading induced by the pair ${\bf g}=(g_1,g_2)\in G^2,$ with $g_1^2=g_2^2$ and endowed with the involutions $\gamma_1,\gamma_2$ and $\gamma_3.$ Then:
\begin{enumerate}
    \item $Id^{(G,*),z}(\mathbb{M}_{{\bf g}}^{\gamma_1})=\langle y_{1,g}y_{2,g}, z_{1,g}z_{2,g}, Id^{(G,*)}(\mathbb{M}_{{\bf g}}^{\gamma_1})\rangle_{T_{(G,*)}};$
    \item $Id^{(G,*),z}(\mathbb{M}_{{\bf g}}^{\gamma_2})=\langle y_{1,1}, Id^{(G,*)}(\mathbb{M}_{{\bf g}}^{\gamma_2})\rangle_{T_{(G,*)}};$
    \item $Id^{(G,*),z}(\mathbb{M}_{{\bf g}}^{\gamma_3})=\langle y_{1,1}, Id^{(G,*)}(\mathbb{M}_{{\bf g}}^{\gamma_3})\rangle_{T_{(G,*)}};$
\end{enumerate}
    
\end{theorem}

In the case where $M_2(F)$ is equipped with the elementary $G$-grading induced by a pair ${\bf g}=(g_1,g_2)\in G^2$, with $g_1^2\neq g_2^2,$ Theorem \ref{thm:elementary} implies that the only graded involutions on $M_2(F)$ are $\gamma_2$ and $\gamma_3.$ In this case, the $G$-grading is given by $M_2(F)^{(1)}=Fe_{11}+Fe_{22}, M_2(F)^{(g)}=Fe_{12}, M_2(F)^{(g^{-1})}=Fe_{21}$ where $g\in G$ satisfies $g^2\ne 1$ and $M_2(F)^{(h)}=\{0\},$ for all $h\neq 1,g,g^{-1}.$

In the next table, we provide the decomposition of $M_2(F)$ into symmetric and skew nonzero homogeneous components, when equipped with an elementary grading and endowed with the involutions $\gamma_2$ and $\gamma_3$.

\begin{table}[ht]\label{tbl:decompositionelem}
\centering
\begin{tabular}{|c|c|c|}
\hline
     {}& $\gamma_2$ & $\gamma_3$  \\ \hline
      $(M_2(F)^{(1)})^+$ & $F(e_{11}+e_{22})$ & $F(e_{11}+e_{22})$ \\ \hline
      $(M_2(F)^{(1)})^-$ & $F(e_{11}-e_{22})$ & $F(e_{11}-e_{22})$ \\ \hline
      $(M_2(F)^{(g)})^+$ & $\{0\}$ & $Fe_{12}$ \\ \hline
      $(M_2(F)^{(g)})^-$ & $Fe_{12}$ & $\{0\}$ \\ \hline
      $(M_2(F)^{(g^{-1})})^+$ & \{0\} & $Fe_{21}$ \\ \hline
      $(M_2(F)^{(g^{-1})})^-$ & $Fe_{21}$ & $\{0\}$ \\ \hline
\end{tabular}\vspace{0.3cm}\caption{$M_2(F)$ with elementary $G$-grading and its graded involutions}
\end{table}
\vspace{-0.5cm}



Looking at Table~\ref{tbl:decompositionelem}, we see that the algebra $\mathbb M^{\gamma_i}_{\bf{g}}$ with $i=2, 3$ have a decomposition closely related to a recently studied algebra. More precisely, in~\cite{BR26} we explored the algebra $M_2(F)$ with the elementary $\Z_2$-grading and transpose superinvolution $$\begin{pmatrix}
			a&b\\c&d
		\end{pmatrix}^{trp}=\begin{pmatrix}
		d&-b\\ c&a
		\end{pmatrix}.$$
The following decomposition is obtained:
\begin{align*}
    (M_2(F)^{(0)})^+=F(e_{11}+e_{22}),&\; (M_2(F)^{(0)})^-=F(e_{11}-e_{22}),  \\ (M_2(F)^{(1)})^+=Fe_{21},&\; (M_2(F)^{(1)})^-=Fe_{12}.
\end{align*}
We consider $ F\langle \overline{X}| *\rangle$, the free associative $*$-algebra with superinvolution  on the variables $\bar{x}_{i, j}$ with $\bar{x}=\bar{y}$ or $\bar{z}$, $i\ge 1$ and $j\in \{0, 1\}$. The variables $\bar{y}_{i, j}$ are symmetric of homogeneous degree $j$, while $\bar{z}_{i, j}$ are skew of homogeneous degree $j$. They denoted the algebra $M_2(F)$ with the elementary $\Z_2$-grading and transpose superinvolution by $\mathbb M$. In particular, when providing the description of the $T_2^*$-ideal $Id^*(\mathbb M)$, they proved the following result (see Definition 4.1, Remark 4.3 and the proof of Theorem 4.2 in~\cite{BR26}).

\begin{lemma}\label{lem:change}
Let $F$ be an infinite field of characteristic $\ne 2$. Let $R\in \mathcal F\langle \overline{X}| *\rangle$ be a multihomogeneous polynomial of degree $n$ such that every monomial appearing in $R$ has one of the following forms:
\begin{enumerate}[(i)]
\item $\alpha\cdot  \bar{y}_{s_1, 0}\cdots \bar{y}_{s_j, 0}\bar{z}_{t_1, 0}\cdots \bar{z}_{t_k, 0}$ with $j+k=n$;
    \item $\alpha\cdot  \bar{y}_{s_1, 0}\cdots \bar{y}_{s_j, 0}\bar{z}_{t_1, 0}\cdots \bar{z}_{t_k, 0}\cdot \bar{y}_{u_1, 1}\bar{z}_{v_1, 1}\cdots \bar{y}_{u_l, 1}\bar{z}_{v_l, 1}$ with $j+k+2l=n$.
\item $\alpha\cdot  \bar{y}_{s_1, 0}\cdots \bar{y}_{s_j, 0}\bar{z}_{t_1, 0}\cdots \bar{z}_{t_k, 0}\cdot \bar{z}_{v_1, 1}\bar{y}_{u_1, 1}\cdots \bar{z}_{v_l, 1}\bar{y}_{u_l, 1}$ with $j+k+2l=n$;
\item $\alpha\cdot  \bar{y}_{s_1, 0}\cdots \bar{y}_{s_j, 0}\bar{z}_{t_1, 0}\cdots \bar{z}_{t_k, 0}\cdot \bar{y}_{u_1, 1}\bar{z}_{v_1, 1}\cdots \bar{y}_{u_l, 1}\bar{z}_{v_l, 1}\bar{y}_{u_{l+1}, 1}$ with $j+k+2l+1=n$;
\item $\alpha\cdot  \bar{y}_{s_1, 0}\cdots \bar{y}_{s_j, 0}\bar{z}_{t_1, 0}\cdots \bar{z}_{t_k, 0}\cdot \bar{z}_{v_1, 1}\bar{y}_{u_1, 1}\cdots \bar{z}_{v_l, 1}\bar{y}_{u_l, 1}\bar{z}_{v_{l+1}, 1}$ with $j+k+2l+1=n$,

\end{enumerate}
where the sequences $s_i, t_i, u_i$ and $v_i$ are non decreasing and in the set $\{1, \ldots, n\}$. If $R\in Id^*(\mathbb M)$, then $R$ equals the zero polynomial.
\end{lemma}
Moreover, when providing the description of the $T_2^*$-subspace of central polynomials for $\mathbb M$, they proved the following result (see Definition 4.1, Remark 5.4 and the proof of Theorem 5.3 in~\cite{BR26}).

\begin{lemma}\label{lem:change-central}
    Let $F$ be an infinite field of characteristic $\ne 2$ and let $R\in \mathcal F\langle \overline{X}| *\rangle$ be as in Lemma~\ref{lem:change}. If $R$ is central, then $R$ is a sum of polynomials of the following forms: $\alpha\cdot \bar{y}_{s_1, 0}\cdots \bar{y}_{s_j, 0}\bar{z}_{t_1, 0}\cdots \bar{z}_{t_{2\ell}, 0}$ and
$$\alpha\cdot \bar{y}_{s_1, 0}\cdots \bar{y}_{s_j, 0}\bar{z}_{t_1, 0}\cdots \bar{z}_{t_k, 0}\cdot (\bar{y}_{u_1, 1}\bar{z}_{v_1, 1}\cdots \bar{y}_{u_l, 1}\bar{z}_{v_l, 1}+(-1)^k\cdot \bar{z}_{v_1, 1}\bar{y}_{u_1, 1}\cdots \bar{z}_{v_l, 1}\bar{y}_{u_l, 1}).$$
\end{lemma}

At first sight, one might be tempted to identify the corresponding variables in the two settings directly. If both algebras were graded by the same group and endowed with involutions, this correspondence could be explored in the context of Lemma~\ref{lem:id-equiv}. Here, however, one algebra is considered with a superinvolution, while the other one is considered with a $G$-graded involution, so some care is needed.

Nevertheless, for the particular algebras under consideration, the admissible evaluations of the corresponding variables range over the exact same subspaces of $M_2(F)$. This makes the two algebras essentially twin from the point of view of polynomial identities, and the same correspondence can also be used for central polynomials, codimensions and cocharacters. We shall explain these connections in detail only for $\mathbb M^{\gamma_2}_{\bf{g}}$. Indeed, the case of $\mathbb M^{\gamma_3}_{\bf{g}}$ follows identically: see the end of this section for more details.

We start with the polynomial identities. 
Consider the following \gi-polynomials:

\begin{enumerate}[(i)]
    \item $y_{1,g}$;
    \item $y_{1,g^{-1}}$;
    \item $z_{1,g}z_{2,g}$;
    \item $z_{1,g^{-1}}z_{2,g^{-1}}$;
    \item $[y_{1,1},x]$ with $x\in \{y_{2, h}, z_{2, h}\}$ and $h=1, g, g^{-1}$;
    \item $z_{1, 1}\circ z_{1, h}$ with $h=g, g^{-1}$;
    \item $[z_{1,1},z_{2,1}]$.
    
\end{enumerate}
 Let $I$ be the $T_{(G,*)}$-ideal generated by the \gi-polynomials above and by $x_{1,h},$ with $h\neq 1,g,g^{-1}$ and $x=y$ or $x=z$. We obtain the following result.

 \begin{theorem}\label{thm:elem-g}
Let $F$ be an infinite field of characteristic $\ne 2$ and let $\mathbb{M}_{\bf g}^{\gamma_2}$ be the algebra $M_2(F)$ equipped with the elementary $G$-grading induced by the pair ${\bf g}= (g_1,g_2),$ with $g_1^2\ne g_2^2,$ and endowed with the involution $\gamma_2$. Then $Id^{(G,*)}(\mathbb{M}_{\bf g}^{\gamma_2})=I$.
 \end{theorem}
\begin{proof}
It is direct to verify that $I\subseteq Id^{(G,*)}(\mathbb{M}_{g}^{\gamma_2})$. Conversely, let $f\in Id^{(G,*)}(\mathbb{M}_{g}^{\gamma_2})$. We can assume that $f$ is multihomogeneous of degree $n\ge 1$. In particular, since $y_{i, h}, z_{i, h}\in I$ for every $h\ne 1, g, g^{-1}$, identities (i) and (ii) imply the following: $f\equiv P\pmod I$, where $P$ is multihomogeneous of degree $n$ and  only contains monomials in the variables $y_{i, 1}, z_{i, 1}, z_{i, g}$ and $z_{i, g^{-1}}$. From the identities in (v), (vi) and (vii), we can also assume that every monomial in $P$ has the form $\alpha\cdot M\cdot N$, where $M$ is of the form 
$$y_{s_1, 1}\cdots y_{s_i, 1}z_{t_1, 1}\cdots z_{t_j, 1},$$
with $1\le s_k\le s_{k+1}\le n$ and $1\le t_k\le t_{k+1}\le n$, and $N$ only contains variables $z_{i, g}$ and $z_{i, g^{-1}}$ with $1\le i\le n$. Finally, from identities (iii) and (iv), we can suppose that, in the monomial $N$, the degrees $g$ and $g^{-1}$ of the variables alternate. Since
\begin{align*}
z_{1,g^{-1}}\circ [z_{1,g},z_{2,g^{-1}}]
=& z_{1,g^{-1}}z_{1,g}z_{2,g^{-1}}
-z_{1,g^{-1}}z_{2,g^{-1}}z_{1,g}\\
&+z_{1,g}z_{2,g^{-1}}z_{1,g^{-1}}
-z_{2,g^{-1}}z_{1,g}z_{1,g^{-1}},
\end{align*}
identities (iii) and (iv) imply
$$z_{1, g^{-1}}\circ [z_{1, g}, z_{2, g^{-1}}]\equiv z_{1, g^{-1}}z_{1, g}z_{2, g^{-1}}-z_{2, g^{-1}}z_{1, g}z_{1, g^{-1}}\pmod I$$
Now observe that $ [z_{1, g}, z_{1, g^{-1}}]\in (\mathcal F^{(1)})^-$. Hence identity (vi) entails that $z_{1, g^{-1}}\circ [z_{1, g}, z_{2, g^{-1}}]\in I$. In conclusion,
\(z_{1, g^{-1}}z_{1, g}z_{2, g^{-1}}\equiv z_{2, g^{-1}}z_{1, g}z_{1, g^{-1}}\pmod I.\)
In a similar way we obtain $z_{1, g}z_{1, g^{-1}}z_{2, g}\equiv z_{2, g}z_{1, g^{-1}}z_{1, g}\pmod I.$
From these two consequences we can also assume that, among the variables of the same degree in $N$, their indices are ordered in increasing order.

Let $Q\in F\langle \overline{X}| *\rangle$ be the polynomial obtained from $P$ by the following change of variables: $y_{i, 1}\to \bar{y}_{i, 0}$, $z_{i, 1}\to \bar{z}_{i, 0}, z_{i, g}\to \bar{z}_{i, 1}$ and $z_{i, g^{-1}}\to \bar{y}_{i, 1}$. From the properties satisfied by the monomials in $P$, it is clear that $Q$ fulfills the conditions in Lemma~\ref{lem:change}.

The variables occurring in $P$ and $Q$ belong to free algebras endowed with different structures. Nevertheless, for the algebras under consideration, the admissible evaluations of the corresponding variables range over the exact same subspaces of $M_2(F)$. More precisely, under the above correspondence of variables, every admissible evaluation of $P$ on $\mathbb M^{\gamma_2}_{\bf{g}}$ determines an admissible evaluation of $Q$ on $\mathbb M$  with the same values in the algebra, and conversely. Therefore, since $P$ is an identity for $\mathbb{M}_{\bf g}^{\gamma_2}$, then $Q\in  Id^*(\mathbb M)$. From Lemma~\ref{lem:change}, it follows that $Q$ is the zero polynomial. From construction, the latter implies that $P$ is the zero polynomial. Since $f\equiv P\pmod I$, we obtain $f\in I$, concluding the proof.\end{proof}

In the previous theorem, some generators of $I$ are redundant. In fact, we included them just to ease the connection with Lemma~\ref{lem:change}. In what follows, we present a smaller set of generators for $Id^{(G,*)}(\mathbb{M}_{\bf g}^{\gamma_i})$ with $i=2, 3$.

\begin{theorem}\label{cor:id-elementary}
    Let $F$ be an infinite field of characteristic $\ne 2$. Then the $T_{(G, *)}$-ideal
    $Id^{(G,*)}(\mathbb{M}_{\bf g}^{\gamma_2})$ is generated by  $$\{x_{1, h}, y_{1, g}, y_{1, g^{-1}}, z_{1, g}z_{2, g}, z_{1, g^{-1}}z_{2, g^{-1}}, [y_{1, 1}, y_{2, 1}], [y_{1, 1}, z_{1, 1}], [z_{1, 1}, z_{2, 1}]\},$$ 
    where $x=y$ or $x=z$ and $h\in G$ with $h\ne 1, g, g^{-1}$.
\end{theorem}
 \begin{proof}
This follows directly from Theorem~\ref{thm:elem-g} and the fact that, for $h=g, g^{-1}$, the polynomials $[y_{1,1}, z_{1, h}]$ and $z_{1, 1}\circ z_{1, h}$ are consequences of $y_{1, h}$.
\end{proof}

\begin{remark}\label{rem:sum}
From the proof of Theorem~\ref{thm:elem-g} we conclude that, modulo $Id^{(G,*)}(\mathbb{M}_{\bf g}^{\gamma_2})$, every multihomogeneous polynomial $P\in \mathcal F$ of degree $n$ is a linear combination of monomials of the following forms

\begin{enumerate}[(i)]
\item $ y_{s_1, 1}\cdots y_{s_j, 1}z_{t_1, 1}\cdots z_{t_k, 1}$ with $j+k=n$;
\item $y_{s_1, 1}\cdots y_{s_j, 1}z_{t_1, 1}\cdots z_{t_k, 1}\cdot z_{u_1, g^{-1}}z_{v_1, g}\cdots z_{u_l, g^{-1}}z_{v_l, g}$ with $j+k+2l=n$.
\item $y_{s_1, 1}\cdots y_{s_j, 1}z_{t_1, 1}\cdots z_{t_k, 1}\cdot z_{v_1, g}z_{u_1, g^{-1}}\cdots z_{v_l, g}z_{u_l, g^{-1}}$ with $j+k+2l=n$;
\item $ y_{s_1, 1}\cdots y_{s_j, 1}z_{t_1, 1}\cdots z_{t_k, 1}\cdot z_{u_1, g^{-1}}z_{v_1, g}\cdots z_{u_l, g^{-1}}z_{v_l, g}z_{u_{l+1}, g^{-1}}$ with $j+k+2l+1=n$;
\item $ y_{s_1, 1}\cdots y_{s_j, 1}z_{t_1, 1}\cdots z_{t_k, 1}\cdot z_{v_1, g}z_{u_1, g^{-1}}\cdots z_{v_l, g}z_{u_l, g^{-1}}z_{v_{l+1}, g}$ with $j+k+2l+1=n$,
\end{enumerate}
where the sequences $s_i, t_i, u_i$ and $v_i$ are non decreasing and in $\{1, \ldots, n\}$.
Moreover, such polynomials are linearly independent modulo the ideal\linebreak $Id^{(G,*)}(\mathbb{M}_{\bf g}^{\gamma_2})$.
\end{remark}

Write $G=\{g_1=1, g_2=g, g_3=g^{-1}, g_4, \ldots, g_k\}$. We proceed to the computation of the $n$-th $(G, *)$-codimension and the $\langle n\rangle$-cocharacter of $\mathbb{M}_{\bf g}^{\gamma_2}$.

\begin{theorem}\label{thm:cod-elementar}
We have $c_n^{(G, *)}(\mathbb{M}_{\bf g}^{\gamma_2})\approx 4^nn^{-1/2}$.
\end{theorem}

\begin{proof}
 By Eq.~\eqref{293-}, we have  \begin{equation}\label{eq:sum-cod}
    c_n^{(G,*)}(\mathbb{M}_{\bf g}^{\gamma_2})=\displaystyle \sum_{\langle n\rangle}\binom{n}{{\langle n\rangle}}c_{\langle n\rangle}(\mathbb{M}_{\bf g}^{\gamma_2}).
\end{equation}
For a composition $\langle n\rangle=n_1+\cdots+n_{2k}$ of $n$, Remark~\ref{rem:sum} implies that $c_{\langle n\rangle}(\mathbb{M}_{\bf g}^{\gamma_2})=0$ unless $|n_4-n_6|\le 1$ and $n_3=n_5=n_k=0$ for every $k\ge 7$. Moreover, assuming that $|n_4-n_6|\le 1$ and $n_3=n_5=n_k=0$ for every $k\ge 7$, Remark~\ref{rem:sum} implies that $c_{\langle n\rangle}(\mathbb{M}_{\bf g}^{\gamma_2})=2$ if $n_4=n_6>0$ and 
$c_{\langle n\rangle}(\mathbb{M}_{\bf g}^{\gamma_2})=1$ in the remaining cases $n_4=n_6=0$ or $|n_4-n_6|=1$. Therefore, Eq.~\eqref{eq:sum-cod} yields 
$$ c_n^{(G, *)}(\mathbb{M}_{\bf g}^{\gamma_2})=\displaystyle \sum_{n_1+n_2=n}\binom{n}{n_1, n_2}+2\displaystyle \sum_{n_1+n_2+n_4+n_6=n\atop n_4-n_6=1\;\text{or}\; n_4=n_6>0}\binom{n}{n_1, n_2, n_4, n_6}.$$
Looking at the proof of Proposition 4.4 in ~\cite{BR26}, the sum above equals $c_n^*(\mathbb M)$ for every $n\ge 1$. According to the same proposition, we have $c_n^*(\mathbb M)\approx 4^nn^{-1/2}$ from where the result follows.
\end{proof}

As in the proof of Theorem~\ref{thm:cod-elementar}, Remark~\ref{rem:sum}  easily yields the following result.

\begin{theorem}\label{thm:coc-elementar}
        Let $F$ be a field of characteristic zero and consider the decomposition $\chi_{\langle n\rangle}(\mathbb{M}_{\bf g}^{\gamma_2})=\displaystyle \sum_{\langle \lambda \rangle\vdash \langle n\rangle} m_{\langle \lambda \rangle}\chi_{\langle \lambda \rangle}$ of the $\langle n\rangle$-cocharacter of $\mathbb{M}_{\bf g}^{\gamma_2}$. Then $m_{\langle \lambda \rangle}=0$ unless $|n_4-n_6|\le 1$ and $n_3=n_5=n_j=0$ for every $7\le j\le 2k$. Moreover, assuming the latter, we have  
        
    $$ m_{\langle \lambda \rangle}=\begin{cases}
         2, &\, \text{if}\;\,  n_4=n_6>0;\\
         1, & \, \text{if}\;\, |n_4-n_6|=1\;\, \text{or}\;\, n_4=n_6=0.\end{cases}$$
    \end{theorem}

We now move to the study of central polynomials for the algebra $\mathbb M^{\gamma_2}_{\bf{g}}$. 
Let $U$ be the $T_{(G, *)}$-subspace generated by $y_{1, 1}$ and the polynomials in $Id^{(G, *)}(\mathbb M^{\gamma_2}_{\bf{g}})$.

\begin{lemma}\label{lem:U}
  The set $U$ is closed under products and contains the polynomials $z_{1,1}z_{2, 1}$  and $z_{i, g}\circ z_{i,g^{-1}}$. Moreover, for integers $k\ge 1$, $u_1\le \cdots\le u_k$ and $v_1\le \cdots \le v_k$, $U$ also contains the polynomials
  $$\prod_{i=1}^kz_{u_i, g}z_{v_i, g^{-1}}+\prod_{i=1}^kz_{v_i, g^{-1}}z_{u_i, g}\quad \text{and}\quad z_{1, 1}\cdot \left(\prod_{i=1}^kz_{u_i, g}z_{v_i, g^{-1}}-\prod_{i=1}^kz_{v_i, g^{-1}}z_{u_i, g}\right).$$
\end{lemma}
\begin{proof}
Write $I=Id^{(G, *)}(\mathbb M^{\gamma_2}_{\bf{g}})$, hence $U$ is generated by $I$ and $y_{1, 1}$. In particular, in order to verify that $U$ is closed under products, we just need to prove that a finite product of variables $y_{i, 1}$ is in $U$. The latter follows directly from the fact that, modulo $I$,  any such product is symmetric and of homogeneous degree $1$. Observe that $z_{1, g}\circ z_{1,g^{-1}}$ is symmetric of degree $1$. Moreover, modulo $I$, the polynomial $z_{1,1}z_{2, 1}$ is also symmetric and of degree $1$ (recall that $[z_{1, 1}, z_{2, 1}]\in I$). In particular, since $y_{1, 1}\in U$, we have $z_{1, g}\circ z_{1,g^{-1}}, z_{1,1}z_{2, 1}\in U$. We now prove the second statement. From the first statement, we have
$$P:=\prod_{i=1}^kz_{ u_i, g}\circ z_{v_i, g^{-1}}\in U.$$
However, since $z_{1,h}z_{2,h}\in I$ for $h=g, g^{-1}$, we easily obtain
$$P\equiv \prod_{i=1}^kz_{u_i, g}z_{v_i, g^{-1}}+\prod_{i=1}^kz_{v_i, g^{-1}}z_{u_i, g}\pmod {I},$$
and so $U$ contains this last polynomial. For the second polynomial, observe that $[z_{1, g}, z_{1, g^{-1}}]$ is skew symmetric of degree $1$ and, 
as $z_{1, 1}z_{2, 1}\in U$, we conclude that $z_{1, 1}\cdot [z_{1, g}, z_{1, g^{-1}}]\in U$. As before, we obtain 
$$Q:=z_{1, 1}\cdot [z_{u_1, g}, z_{v_1, g^{-1}}]\prod_{i=2}^kz_{u_i, g}\circ z_{v_i, g^{-1}}\in U,$$
and $$Q\equiv z_{1, 1}\cdot \left(\prod_{i=1}^kz_{u_i, g}z_{v_i, g^{-1}}-\prod_{i=1}^kz_{v_i, g^{-1}}z_{u_i, g}\right)\pmod {I},$$
concluding the proof. 
\end{proof}

We are now ready to prove the following result. 

\begin{theorem}\label{thm:central-elementar}
    Let $F$ be an infinite field of characteristic $\ne 2$. Then $Id^{(G, *), z}(\mathbb M^{\gamma_2}_{\bf{g}})=U$.
\end{theorem}
\begin{proof}
We directly verify that every element of $U$ is central. Conversely, fix $P$ a central polynomial. Without loss of generality, we can assume that $P$ is a multihomogeneous polynomial of degree $n$. We observe that any polynomial $P'$ with $P\equiv P'\pmod {Id^{(G, *)}(\mathbb M^{\gamma_2}_{\bf{g}}})$ is also central. In particular, according to Remark~\ref{rem:sum}, we can assume that $P$ is a linear combination of monomials of the forms (i)-(v) as presented in Remark~\ref{rem:sum}.  As in the proof of Theorem~\ref{thm:elem-g}, let $Q\in F\langle \overline{X}| *\rangle$ be the polynomial obtained from $P$ by the following change of variables: $y_{i, 1}\to \bar{y}_{i, 0}$, $z_{i, 1}\to \bar{z}_{i, 0}, z_{i, g}\to \bar{z}_{i, 1}$ and $z_{i, g^{-1}}\to \bar{y}_{i, 1}$. From the properties satisfied by the monomials in $P$, it is clear that $Q$ fulfills the conditions in Lemma~\ref{lem:change}. 

Recall that, from hypothesis, $P$ is a central polynomial for $\mathbb M^{\gamma_2}_{\bf{g}}$. Following the same arguments in the proof of Theorem~\ref{thm:elem-g}, we conclude that $Q$ is a central polynomial for $\mathbb M$, hence $Q$ also fulfills the conditions in Lemma~\ref{lem:change-central}. Therefore, from Lemma~\ref{lem:change-central} and the correspondence between $P$ and $Q$, we conclude that $P$ is a sum of polynomials of the following forms: $\alpha\cdot y_{s_1, 1}\cdots y_{s_j, 1}z_{t_1, 1}\cdots z_{t_{2\ell}, 1}$ and

$$
\alpha\cdot y_{s_1, 1}\cdots y_{s_j, 1}z_{t_1, 1}\cdots z_{t_k, 1}\cdot
\left(
\prod_{i=1}^lz_{u_i, g}z_{v_i, g^{-1}}
+(-1)^k\prod_{i=1}^l z_{v_i, g^{-1}}z_{u_i, g}\right).
$$
From Lemma~\ref{lem:U}, polynomials of these forms belong to $U$. Thus $P\in U$, concluding the proof.   
\end{proof}

We end this section with the analogous results for the algebra $\mathbb M_{\bf{g}}^{\gamma_3}$. The correspondence previously established between $\mathbb M_{\bf{g}}^{\gamma_2}$ and $\mathbb M$ applies in the same way to the pair $(\mathbb M,\mathbb M_{\bf{g}}^{\gamma_3})$, with only a different identification of the variables. More precisely, in this case we consider

$$
y_{i,1}\to \bar{y}_{i,0},\qquad
z_{i,1}\to \bar{z}_{i,0},\qquad
y_{i,g}\to \bar{z}_{i,1},\qquad
y_{i,g^{-1}}\to \bar{y}_{i,1}.
$$

Under this change of variables, the corresponding admissible evaluations range over the same subspaces as in the previous case. Consequently, all the arguments used for $\mathbb M_{\bf{g}}^{\gamma_2}$ carry over verbatim after replacing the variables according to the above correspondence. Theorems~\ref{cor:id-elementary},~\ref{thm:cod-elementar},~\ref{thm:coc-elementar} and~\ref{thm:central-elementar} immediately yield the following result.

\begin{theorem}
Let $F$ be an infinite field of characteristic $\ne 2$. Then the following hold:

\begin{enumerate}[(a)]
    \item The $T_{(G, *)}$-ideal
    $Id^{(G,*)}(\mathbb{M}_{\bf g}^{\gamma_3})$ is generated by $$\{x_{1, h}, z_{1, g}, z_{1, g^{-1}}, y_{1, g}y_{2, g}, y_{1, g^{-1}}y_{2, g^{-1}}, [y_{1, 1}, y_{2, 1}], [y_{1, 1}, z_{1, 1}], [z_{1, 1}, z_{2, 1}]\},$$
where $x=y$ or $x=z$ and $h\in G$ with $h\ne 1, g, g^{-1}$.    
   
    \item $c_n^{(G, *)}(\mathbb{M}_{\bf g}^{\gamma_3})\approx 4^nn^{-1/2}$.

\item Assume that $F$ has characteristic zero and consider the decomposition $\chi_{\langle n\rangle}(\mathbb{M}_{\bf g}^{\gamma_3})=\displaystyle \sum_{\langle \lambda \rangle\vdash \langle n\rangle} m_{\langle \lambda \rangle}\chi_{\langle \lambda \rangle}$ of the $\langle n\rangle$-cocharacter of $\mathbb{M}_{\bf g}^{\gamma_3}$. Then $m_{\langle \lambda \rangle}=0$ unless $|n_3-n_5|\le 1$ and $n_4=n_6=n_j=0$ for every $7\le j\le 2k$. Moreover, assuming the latter, we have  
        
    $$ m_{\langle \lambda \rangle}=\begin{cases}
         2, &\, \text{if}\;\,  n_3=n_5>0;\\
         1, & \, \text{if}\;\, |n_3-n_5|=1\;\, \text{or}\;\, n_3=n_5=0.\end{cases}$$

          \item   $Id^{(G, *), z}(\mathbb M^{\gamma_3}_{\bf{g}})=\langle y_{1, 1}, Id^{(G, *)}(\mathbb M^{\gamma_3}_{\bf{g}})\rangle^{T_{(G, *)}}$.
    
\end{enumerate}

\end{theorem}

\section{$M_2(F)$ with $(-1)$-graded involution}

In this section, we deal with the algebra $M_2(F)$ when it is equipped with the $(-1)$-grading and endowed with the involutions $\gamma_i, 1\leq i\leq 4$, as in Theorem \ref{-1}. We write the Klein group $K=\{1, a, b, ab\}$. In Table~\ref{decomposition-1} we present the decomposition of $M_2(F)$ into symmetric and skew homogeneous components, when it is equipped with $(-1)$-grading and endowed with involution $\gamma_i$.

\begin{table}[h!]\label{decomposition-1}
\centering
\begin{tabular}{|c|c|c|c|c|}
\hline
{}& $\gamma_1$& $\gamma_2$& $\gamma_3$& $\gamma_4$\\ \hline
 $(M_2(F)^{(1)})^+$& $F(e_{11}+e_{22})$& $F(e_{11}+e_{22})$& $F(e_{11}+e_{22})$& $F(e_{11}+e_{22})$\\ \hline
 $(M_2(F)^{(1)})^-$& \{0\}& \{0\}& \{0\}& \{0\}\\ \hline
 $(M_2(F)^{(a)})^+$& $F(e_{12}+e_{21})$& \{0\}& $F(e_{12}+e_{21})$& \{0\}\\ \hline
 $(M_2(F)^{(a)})^-$& \{0\}& $F(e_{12}+e_{21})$& \{0\}& $F(e_{12}+e_{21})$ \\ \hline
 $(M_2(F)^{(b)})^+$& $F(e_{11}-e_{22})$& \{0\}& \{0\}& $F(e_{11}-e_{22})$\\ \hline
 $(M_2(F)^{(b)})^-$& \{0\}& $F(e_{11}-e_{22})$& $F(e_{11}-e_{22})$& \{0\}\\ \hline
 $(M_2(F)^{(ab)})^+$& \{0\}& \{0\}& $F(e_{12}-e_{21})$& $F(e_{12}-e_{21})$ \\ \hline
 $(M_2(F)^{(ab)})^-$& $F(e_{12}-e_{21})$& $F(e_{12}-e_{21})$& \{0\}& \{0\}\\ \hline
\end{tabular}\vspace{0.3cm}\caption{$M_2(F)$ with $(-1)$-grading and its graded involutions}
\end{table}

Let $\mathbb M_{-1}^{\gamma_i}$ denote the algebra $M_2(F)$ endowed with involution $\gamma_i, 1\leq i\leq 4,$ and equipped with the $(-1)$-grading.
   
\begin{remark}\label{rem:id-equiv-klein}
 Following the notation of Lemma~\ref{lem:id-equiv}, we observe that the algebras $A_1=\mathbb M^{\gamma_1}_{-1}$ and $A_2=\mathbb M^{\gamma_2}_{-1}$ satisfy all the conditions in Lemma~\ref{lem:id-equiv}. Indeed, both are endowed with $K$-graded involutions and, for every $g\in K$, we have $\{(A_1^{(g)})^{\pm}\}=\{(A_2^{(g)})^{\pm}\}$. Moreover, if $\varepsilon_g\in \Z_2$ is as in Lemma~\ref{lem:id-equiv}, we have $\varepsilon_1=\varepsilon_{ab}=0$ and $\varepsilon_a=\varepsilon_b=1$. In particular, the map $g\mapsto \varepsilon_g$ is a group homomorphism. The same can be verified for the cases $A_2=\mathbb M^{\gamma_i}_{-1}$ and $i=3, 4$. In particular, from Lemma~\ref{lem:id-equiv} and Corollary~\ref{cor:equal}, it suffices to study the algebra $\mathbb M^{\gamma_1}_{-1}$.
\end{remark}


\subsection{$(K,*)$-identities and the $\langle n \rangle$-cocharacter} In this subsection we shall present the generators of the $T_{(K,*)}$-ideal of identities of $M_2(F)$ when it is endowed with the $K$-graded involutions $\gamma_i, 1\leq i\leq 4.$ Moreover, we explicitly compute the sequence $c_{n}^{(G,*)}(M_2(F))$ in these cases and present the sequence of $\langle n\rangle$-cocharacters of $M_2(F)$ in each case. We begin with the following definition.

\begin{definition}\label{def:notation}
  For a partition $S\sqcup T\sqcup U\sqcup V=\{1,\ldots,n\}$, set
\[
(YZ)_{S,T,U,V}
:=
\prod_{s\in S} y_{s,1}\,
\prod_{t\in T}y_{t,a}\,
\prod_{u\in U} y_{u,b}\,
\prod_{v\in V}z_{v,ab}\in F\langle X|*\rangle ,
\]
where each product is taken with the indices in increasing order (with the convention that the product is one if the corresponding set is empty).
\end{definition}

\begin{theorem}\label{thm:id-klein}
    Let $F$ be a field of characteristic zero. Then the $T_{(K,*)}$-ideal $I=Id^{(K,*)}(\mathbb M_{-1}^{\gamma_1})$ is generated by the following polynomials:

    \begin{enumerate}
    \item  $y_{1, ab}, z_{1, 1}, z_{1, a}$ and $ z_{1, b}$;
			\item $[y_{1, 1}, x]$ with $x\in \{y_{2, 1}, y_{2, a}, y_{2, b}, z_{2, ab}\}$;
 \item $y_{1, a}\circ y_{1, b}$;
 \item $y_{1, g}\circ z_{1, ab}$ with $g=a, b$. 
		\end{enumerate}
\end{theorem}
\begin{proof}
Let $J$ be the $T_{(K,*)}$-ideal generated by the above polynomials. It is direct to verify that $J\subseteq I$. Now let $f\in I$. Since $F$ has characteristic $0$, we can assume that $f$ is multilinear of degree $n\ge 1$. From the identities in (1), we see that $f\equiv P\pmod J$, where $P$ is a multilinear polynomial of degree $n$ without monomials with any of the variables $y_{1, ab}, z_{1, 1}, z_{1, a}$ and $ z_{1, b}$. Moreover, from identities (2), (3) and (4), we can also assume that any monomial appearing in $P$ has the form $\alpha \cdot M_1\cdot M_a\cdot M_b\cdot M_{ab}$, where $M_g$ only contains variables of the type $y_{i, g}$ if $g\in \{1, a, b\}$ and $M_{ab}$ only contains variables of the type $z_{i, ab}$. 

Since $g^2=1$ for every $g\in K$, the commutator $[w_1,w_2]$ is skew-symmetric and homogeneous of degree $1$ whenever $w_1$ and $w_2$ are variables of the same symmetry and homogeneous degree. Hence, $[w_1,w_2]$ is a consequence of the identity $z_{1,1}$. Therefore, we may assume that the variables in each $M_g$ are ordered according to their indices. With Definition~\ref{def:notation} in mind, it follows that every monomial in $P$ has the form $\alpha(YZ)_{S,T,U,V}$.

We now compute the image of each such monomial with an arbitrary admissible $(K, *)$-evaluation $\eta$: $y_{i, 1}=a_i(e_{11}+e_{22}), z_{i, 1}=0, y_{i, a}=b_i(e_{12}+e_{21}), z_{i, a}=0, y_{i, b}=c_i(e_{11}-e_{22}), z_{i, b}=0, y_{i, ab}=0, z_{i, ab}=d_i(e_{12}-e_{21})$. A direct computation yields
\begin{equation}\label{eq:evaluation}
    \eta((YZ)_{S,T,U,V})=\pi\cdot (-1)^{\left\lfloor\frac{|V|}{2}\right\rfloor}\cdot (e_{12}+e_{21})^{e_T}(e_{11}-e_{22})^{e_U}(e_{12}-e_{21})^{e_V},
\end{equation}
where $$\pi=\prod_{s\in S}a_s\prod_{t\in T}b_t\prod_{u\in U}c_u\prod_{v\in V}d_v,$$
and $e_T, e_U, e_V\in \{0, 1\}$ are the reductions $\pmod 2$ of the numbers $|T|, |U|$ and $|V|$, respectively.

In particular, it is clear that the parities of $|T|$, $|U|$, $|V|$ and $\left\lfloor\frac{|V|}{2}\right\rfloor$ play a crucial role in the evaluations. For each $m\in\{0,\ldots,7\}$, write
\[
m=m_0\cdot 2^0+m_1\cdot 2^1+m_2\cdot  2^2,
\qquad m_0,m_1,m_2\in\{0,1\}.
\]
We then have
\[
P=P_0+P_1+\cdots+P_7,
\]
where $P_m$ consists of the monomials $\alpha(YZ)_{S,T,U,V}$ of $P$ satisfying
\[
|T|\equiv m_0\pmod 2,\qquad
|U|\equiv m_1\pmod 2,\qquad
|V|\equiv m_2\pmod 2.
\]
We further write $P_{i}=P_{i, +}+P_{i, -}$, where $P_{i, +}$ and $P_{i, -}$ collect the monomials with $\left\lfloor\frac{|V|}{2}\right\rfloor$ being even and odd, respectively.

We want to prove that each $P_{m}$ vanishes. The idea is to take an arbitrary admissible evaluation and employ Lemma~\ref{lem:com}. 

From hypothesis, $P$ is an identity. Take $y_{i, 1}=a_i(e_{11}+e_{22}), z_{i, 1}=0, y_{i, a}=b_i(e_{12}+e_{21}), z_{i, a}=0, y_{i, b}=c_i(e_{11}-e_{22}), z_{i, b}=0, y_{i, ab}=0, z_{i, ab}=d_i(e_{12}-e_{21})$, where the $a_i, b_i, c_i, d_i$'s are arbitrary elements of $F$. If $Q_{i}$ denotes the image of $(P_{i, +}-P_{i, -})^C$ at the elements $a_i, b_i, c_i, d_i$, \eqref{eq:evaluation} yields the following:
\begin{align*}0&=(e_{11}+e_{22})Q_0+(e_{12}+e_{21})Q_1+(e_{11}-e_{22})Q_2+(e_{21}-e_{12})Q_3\\{}&+(e_{12}-e_{21})Q_4+(e_{22}-e_{11})Q_5+(e_{12}+e_{21})Q_6+(e_{11}+e_{22})Q_7.\end{align*}
Since the set $\{e_{11}, e_{12}, e_{21}, e_{22}\}$ is linearly independent in the vector space $M_2(F)$, the last equality yields the following linear system:
$$\begin{cases}
    Q_0+Q_2-Q_5+Q_7=0\\Q_1-Q_3+Q_4+Q_6=0\\ Q_1+Q_3-Q_4+Q_6=0\\ Q_0-Q_2+Q_5+Q_7=0. 
\end{cases}$$
We solve the system and obtain  $$Q_0+Q_7=Q_1+Q_6=Q_2-Q_5=Q_3-Q_4=0.$$ 
From Lemma~\ref{lem:vanish}, we conclude that 
\begin{equation}\label{eq:vanish}
    ((P_{i, +}-P_{i, -})+\varepsilon_i\cdot (P_{7-i, +}-P_{7-i, -}))^C=0,\, 0\le i\le 3,
\end{equation}
where $\varepsilon_0=\varepsilon_1=1\in F$  and $\varepsilon_2=\varepsilon_3=-1\in F$. From construction, it is clear that the polynomials $(P_{i, +}-P_{i, -})\pm (P_{j, +}-P_{j, -})$ do not contain two distinct monomials $M, N$ with $M^C=N^C$ if $i\ne j$. Therefore, \eqref{eq:vanish}
combined with Lemma~\ref{lem:com} implies that 
$$(P_{i, +}-P_{i, -})+\varepsilon_i\cdot (P_{7-i, +}-P_{7-i, -})=0.$$ 
However, it follows by the construction of the $P_{i, \pm}$'s that no monomial cancellation can occur in the last equality. In particular, we obtain $P_{j, \pm}=0$ for every $0\le j\le 7$. Thus $P_m=0$ for every $0\le m\le 7$ and so $P=0$. Since $f\equiv P\pmod J$, it follows that $f\in J$. In conclusion, $I\subseteq J$ and this completes the proof.
\end{proof}

As follows, we provide a minimal set of generators for $Id^{(K,*)}(\mathbb M_{-1}^{\gamma_1})$.

\begin{theorem}\label{cor:klein}
    Let $F$ be a field of characteristic zero. Then 
    $$Id^{(K,*)}(\mathbb M_{-1}^{\gamma_1})=\langle y_{1, ab}, z_{1, 1}, z_{1, a}, z_{1, b}\rangle_{T_{(K,*)}}.$$  
\end{theorem}
\begin{proof}
    Observe that, for every $g\in K$, the element $[y_{1, 1}, y_{2, g}]$ is skew symmetric of homogeneous degree $g$. In particular, this element is a consequence of $z_{1, g}$ if $g\ne ab$. Moreover, $y_{1, a}\circ y_{1, b}$ and $[y_{1, 1}, z_{1, ab}]$ are symmetric of homogeneous degree $ab$, hence both are consequences of $y_{1, ab}$. Finally, for $g=a, b$, the element $y_{1, g}\circ z_{1, ab}$ is skew symmetric of degree $a$ or $b$, hence a consequence of $z_{1, a}$ or $z_{1, b}$. These observations combined with Theorem~\ref{thm:id-klein} complete the proof.
\end{proof}

Using Lemma~\ref{lem:id-equiv} and Remark~\ref{rem:id-equiv-klein}, 
we see that Theorem~\ref{cor:klein} yields the following result.
\begin{theorem}\label{thm:gamma_i}
    Let $F$ be a field of characteristic zero. Then the following holds:
\begin{enumerate}[(a)]
    \item $Id^{(K,*)}(\mathbb M_{-1}^{\gamma_2})=\langle y_{1, a}, y_{1, b}, y_{1, ab}, z_{1, 1}\rangle_{T_{(K,*)}}$;
    \item $Id^{(K,*)}(\mathbb M_{-1}^{\gamma_3})=\langle y_{1, b}, z_{1, 1}, z_{1, a}, z_{1, ab}\rangle_{T_{(K,*)}}$;
    \item $Id^{(K,*)}(\mathbb M_{-1}^{\gamma_4})=\langle y_{1, a}, z_{1, 1}, z_{1, b}, z_{1, ab}\rangle_{T_{(K,*)}}$.
\end{enumerate} 
\end{theorem}

\begin{remark}\label{rem:homo-klein}
In the proof of Theorem~\ref{thm:id-klein} we assumed the field $F$ has characteristic zero in order to work only with multilinear polynomials. However, the proof is essentially the same if we assume that $f\in I$ is not multilinear. In fact, in this case, we just need to extend the definition of $(YZ)_{S, T, U, V}$ by allowing $S, T, U$ and $V$ to be {\em multisets} of $\{1, \ldots, n\}$. We then use the same decomposition $P=P_0+\cdots+P_7$; the resolution of the system of equations on the $Q_i$'s is still valid if $F$ has characteristic $\ne 2$, and Lemma~\ref{lem:com} remains applicable in this new setting. Therefore, Theorems~\ref{thm:id-klein}, ~\ref{cor:klein} and~\ref{thm:gamma_i} hold over any infinite field of characteristic different from $2$.
\end{remark}

Theorem~\ref{thm:id-klein} readily yields the $n$-th \gi-codimension of $\mathbb M_{-1}^{\gamma_1}$.

\begin{corollary}\label{cor:cod}
For each $n\ge 1$ and each $i\in \{1, 2, 3, 4\}$, $c_n^{(K,*)}(\mathbb M_{-1}^{\gamma_i})=4^n$.
\end{corollary}

\begin{proof}
From Remark~\ref{rem:id-equiv-klein} and Corollary~\ref{cor:equal}, it suffices to consider the case $i=1$. We order the elements of $K$ as $K=\{g_1=1, g_2=a, g_3=b, g_4=ab\}$ and, for each composition $\langle n\rangle=(n_1, \ldots, n_8)$ of $n\ge 1$, $P_{\langle n\rangle}(\mathbb M_{-1}^{\gamma_1})$ denotes the corresponding quotient space of multilinear $(K, *)$-polynomials as in Section 3.
Eq.~\eqref{293-} readily yields \begin{equation}\label{codim}
    c_n^{(K, *)}(\mathbb M_{-1}^{\gamma_1})=\displaystyle \sum_{\langle n\rangle}\binom{n}{{\langle n\rangle}}c_{\langle n\rangle}(\mathbb M_{-1}^{\gamma_1}).
\end{equation}
From Theorem~\ref{cor:klein}, we know that $P_{\langle n\rangle}(\mathbb M_{-1}^{\gamma_1})$ is the null space unless $n_2=n_4=n_6=n_7=0$. Moreover, if the latter holds, we follow the proof of Theorem~\ref{thm:id-klein} and conclude that the quotient space $P_{\langle n\rangle}(\mathbb M_{-1}^{\gamma_1})$ is generated by the monomial
 $$y_{1, 1}\cdots y_{n_1, 1}\cdot y_{1, a}\cdots y_{n_3, a}\cdot y_{1, b}\cdots y_{n_5, b}\cdot z_{1, ab}\cdots z_{n_8, ab}.$$
So each such composition contributes with $1$ dimension. Therefore, Eq.~\eqref{codim} gives the equality 
$$c_n^{(K, *)}(\mathbb M_{-1}^{\gamma_1})=\sum_{n_1+n_3+n_5+n_8=n}\binom{n}{n_1, n_3, n_5, n_8}=4^n.$$
\end{proof}

    Next, we will deal with the decomposition (\ref{cocharacter}) of $\chi_{\langle n\rangle}(\mathbb M_{-1}^{\gamma_1})$. 
    We observe that, by Remark~\ref{altura}, for every composition $\langle n \rangle=(n_1,\ldots ,n_8)$ of $n$ and every multipartition $\langle \lambda\rangle\vdash \langle n \rangle$, the corresponding Young multitableaux $T_{\langle \lambda\rangle}$ satisfies the following property: all the tableaux in $T_{\langle \lambda\rangle}$ have at most 1 row. 
    From the proof of Corollary \ref{cor:cod}, we obtain the following. If $n_2=n_4=n_6=n_7=0$ then, modulo $Id^{(K,*)}(\mathbb M_{-1}^{\gamma_1})$, there exists a unique highest weight vector that is not a $\ast$-identity, and hence $m_{\langle \lambda \rangle}=1$. In all other cases, $m_{\langle \lambda \rangle}=0$. This yields the following theorem.

    \begin{theorem}\label{thm:coc}
        Let $F$ be a field of characteristic zero and consider the decomposition $\chi_{\langle n\rangle}(\mathbb M_{-1}^{\gamma_1})=\displaystyle \sum_{\langle \lambda \rangle\vdash \langle n\rangle} m_{\langle \lambda \rangle}\chi_{\langle \lambda \rangle}$ of the $\langle n\rangle$-cocharacter of $\mathbb M_{-1}^{\gamma_1}$. Then, we have  $$ m_{\langle \lambda \rangle}=\begin{cases}
         1, & \, \text{if}\;\, n_2=n_4=n_6=n_7=0,\\
         0, & \, \text{otherwise}.
        \end{cases}$$
    \end{theorem}

    \begin{remark}
       Although Corollary~\ref{cor:equal} does not address the $\langle n\rangle$-cocharacters for $(G, *)$-algebras with the same skew and symmetric homogeneous components, the same correspondence of variables induces a one-to-one correspondence between their multiplicities, after the corresponding components of the compositions of $n$ are interchanged. For instance, the symmetric and skew homogeneous components of $\mathbb M_{-1}^{\gamma_2}$ and $\mathbb M_{-1}^{\gamma_1}$ are interchanged only in the homogeneous degrees $a$ and $b$. Thus, the multiplicities for $\mathbb M_{-1}^{\gamma_2}$ are obtained directly from those of $\mathbb M_{-1}^{\gamma_1}$ by exchanging the corresponding components. In particular, for $\mathbb M_{-1}^{\gamma_2}$, we have
$$ m_{\langle \lambda \rangle}=\begin{cases}
         1, & \, \text{if}\;\, n_2=n_3=n_5=n_7=0,\\
         0, & \, \text{otherwise}.\end{cases}$$
         A similar result holds for $\mathbb M_{-1}^{\gamma_i}$ with $i=3, 4$.
    \end{remark}

\subsection{Central Polynomials} In this subsection we shall present a set of generators for $Id^{(K,*),z}(\mathbb M_{-1}^{\gamma_i})$ with $i=1, 2, 3, 4$. Thanks to Lemma~\ref{lem:id-equiv}, Corollary~\ref{cor:equal} and Remark~\ref{rem:id-equiv-klein}, we only need to deal with the case $i=1$. Let $W$ be the $T_{(K,*)}$-subspace generated by $y_{1, 1}$ and the polynomials in $I=Id^{(K, *)}(\mathbb M_{-1}^{\gamma_1})$.

\begin{lemma}\label{lem:central-klein}
 The set $W$ is closed under products and contains the polynomials $y_{1, a}y_{1, b}z_{1, ab}, y_{1, a}y_{2, a}, y_{1, b}y_{2, b}$ and $z_{1, ab}z_{2, ab}$. In particular, using the notation of Definition~\ref{def:notation}, for every $\alpha\in F$ the set $W$ contains the monomials 
$$\alpha\cdot (YZ)_{S, T, U, V},$$
for which $|T|, |U|$ and $|V|$ have the same parity.
\end{lemma}
\begin{proof}
Recall that $W$ is the $T_{(K,*)}$-space generated by $I$ and $y_{1, 1}$. Therefore, $W$ is closed under products if any finite product of variables $y_{i, 1}$ belongs to $W$. The latter is true since any such product is again symmetric and of homogeneous degree 1, modulo $I$. Now observe that, modulo $I$, the polynomials $y_{1, a}y_{2, a},y_{1, b}y_{2, b}$ and $z_{1, ab}z_{2, ab}$ are symmetric and of homogeneous degree $1$, hence they belong to $W$. The same can be verified for $y_{1, a}y_{1, b}z_{1, ab}$. This completes the proof of the first statement. For the second statement observe that, modulo $I$, the monomials $\alpha\cdot (YZ)_{S, T, U, V}$ with the above restrictions are simply products of polynomials of the form $y_{i, 1}, y_{i, a}y_{j, b}z_{k, ab}, y_{i, a}y_{j, a}, y_{i, b}y_{j, b}$ and $z_{i, ab}z_{j, ab}$. The result then follows from the fact that $W$ contains $I$ and it is closed under products.
\end{proof}

We are ready to prove the following result.

\begin{theorem}\label{thm:klein-central}
      Let $F$ be a field of characteristic zero. Then  $Id^{(K,*),z}(\mathbb M_{-1}^{\gamma_1})=W$.
\end{theorem}

\begin{proof}
We directly verify that any element of $W$ is central. Let $P$ be a central polynomial. Since $F$ has characteristic zero, we can suppose that $P$ is multilinear of degree $n$. It is clear that any polynomial $P'$ with $P\equiv P'\pmod I$ is also central. In particular, we can suppose that $P=P_0+\cdots+P_7$, where the monomials appearing in each $P_i$ are restricted as in the proof of Theorem~\ref{thm:id-klein}. We further write $P_{i}=P_{i, +}+P_{i, -}$ as in the proof of 
Theorem~\ref{thm:id-klein}. From Lemma~\ref{lem:central-klein}, we clearly have $P_0, P_7\in W$. Hence we can assume, without loss of generality, that $P_0=P_7=0$. In particular, it suffices to prove that $P$ vanishes.

Take $y_{i, 1}=a_i(e_{11}+e_{22}), z_{i, 1}=0, y_{i, a}=b_i(e_{12}+e_{21}), z_{i, a}=0, y_{i, b}=c_i(e_{11}-e_{22}), z_{i, b}=0, y_{i, ab}=0, z_{i, ab}=d_i(e_{12}-e_{21})$, where the $a_i, b_i, c_i, d_i$'s are arbitrary elements of $F$. If $R_i$ denotes the image of $(P_{i, +}- P_{i, -})^C$ at the elements $a_i, b_i, c_i, d_i$, as in the proof of Theorem~\ref{thm:id-klein}, we obtain the following:
\begin{align*}R&=(e_{12}+e_{21})R_1+(e_{11}-e_{22})R_2+(e_{21}-e_{12})R_3\\{}&+(e_{12}-e_{21})R_4+(e_{22}-e_{11})R_5+(e_{12}+e_{21})R_6.\end{align*}
From hypothesis, $R$ is a central element of $M_2(F)$ and so Remark~\ref{rem:center} implies that $R=\beta\cdot (e_{11}+e_{22})$ for some $\beta\in F$. The latter yields the following system of equations:
$$\begin{cases}
R_2-R_5=-R_2+R_5\\
R_1-R_3+R_4+R_6=0\\
R_1+R_3-R_4+R_6=0.
\end{cases}$$
Since $F$ has characteristic zero, we obtain $R_2-R_5= R_1+R_6=R_3-R_4=0$. As in the proof of Theorem~\ref{thm:id-klein}, the latter implies that $P_i=0$ for every $1\le i\le 6$. Thus $P=0$, concluding the proof.
\end{proof}

As in the study of polynomial identities for $Id^{(K,*)}(\mathbb M_{-1}^{\gamma_i})$, Lemma~\ref{lem:id-equiv} and Theorem~\ref{thm:klein-central} imply the following result. 
\begin{theorem}\label{thm:central-gamma_i}
    Let $F$ be a field of characteristic zero. Then $$Id^{(K,*),z}(\mathbb M_{-1}^{\gamma_i})=\langle y_{1, 1}, Id^{(K,*)}(\mathbb M_{-1}^{\gamma_i})\rangle^{T_{(K, *)}}, i=1, 2, 3, 4.$$
\end{theorem}

\begin{remark}
    Following Remark~\ref{rem:homo-klein} and the proof of Theorem~\ref{thm:klein-central}, we also conclude that Theorem~\ref{thm:klein-central} (hence Theorem~\ref{thm:central-gamma_i})  holds for any infinite field of characteristic $\ne 2$. 
\end{remark}

As a consequence of Theorem~\ref{thm:klein-central}, we obtain exact formulas for the central and proper central codimensions.

\begin{proposition}
For every $n\ge 1$ and every $i\in \{1, 2, 3, 4\}$, we have $c_{n}^{(K,*),\delta}( \mathbb M_{-1}^{\gamma_i})=4^{n-1}$ and $c_{n}^{(K,*),z}( \mathbb M_{-1}^{\gamma_i})=3\cdot 4^{n-1}$.     
\end{proposition}

\begin{proof}
From Corollary~\ref{cor:equal} and Remark~\ref{rem:id-equiv-klein}, it suffices to consider the case $i=1$. Observe that 
$c_{n}^{(G,*),z}( \mathbb M_{-1}^{\gamma_1})=c_{n}^{(G,*)}( \mathbb M_{-1}^{\gamma_1})-c_{n}^{(G,*),\delta}( \mathbb M_{-1}^{\gamma_1})$. In particular, from Corollary~\ref{cor:cod}, it suffices to prove that $c_{n}^{(G,*),\delta}( \mathbb M_{-1}^{\gamma_1})=4^{n-1}$. Let $\langle n\rangle =n_1+\cdots+n_8$ be a composition of $n$. Following the proof of Theorem~\ref{thm:klein-central}, we observe that $P_{\langle n \rangle}^{\delta}(\mathbb M_{-1}^{\gamma_1})$ is the null space unless $n_2=n_4=n_6=n_7=0$ and $n_3, n_5, n_8$ have the same parity. Moreover, in this case, we see that the corresponding space is generated by the polynomial 
$$y_{1, 1}\cdots y_{n_1, 1}\cdot y_{1, a}\cdots y_{n_3, a}\cdot y_{1, b}\cdots y_{n_5, b}\cdot z_{1, ab}\cdots z_{n_8, ab}.$$
In conclusion, $c_{\langle n \rangle}^{\delta}( \mathbb M_{-1}^{\gamma_1})=1$ if $n_2=n_4=n_6=n_7=0$ and $n_3, n_5, n_8$ have the same parity and, in every other case, this quantity equals $0$. Thus, using Remark~\ref{rmk:codimcentral}, we obtain
$$c_{n}^{(G,*),\delta}( \mathbb M_{-1}^{\gamma_1})=\sum_{n_1+n_3+n_5+n_8=n\atop n_3\equiv n_5\equiv n_8\pmod 2}\binom{n}{n_1, n_3, n_5, n_8}.$$
We shall obtain a closed formula for this last sum. We clearly have the following equality in the commutative ring $\mathbb R[X_1, X_2, X_3, X_4]$: 
$$(X_1+X_2+X_3+X_4)^n=\sum_{a+b+c+d=n}\binom{n}{a, b, c, d}X_1^aX_2^bX_3^cX_4^d.$$
Moreover, for integers $b, c, d\ge 0$, we have 
$$\sum_{e_2, e_3, e_4\in \{\pm 1\}}(1+e_2e_3e_4)e_2^be_3^ce_4^d=\begin{cases}
8,&\, \text{if}\;\; b\equiv c\equiv d\pmod 2,\\
0,&\, \text{otherwise.}
\end{cases}$$
Combining these observations, we easily obtain
\begin{align*}
    R(X_1, X_2, X_3, X_4)&:=\sum_{e_2, e_3, e_4\in \{\pm 1\}} (1+e_2e_3e_4)(X_1+e_2X_2+e_3X_3+e_4X_4)^n\\{}&=8\sum_{a+b+c+d=n\atop b\equiv c\equiv d\pmod 2}\binom{n}{a, b, c, d}X_1^aX_2^bX_3^cX_4^d.
\end{align*}
Evaluating the last identity at $X_1=X_2=X_3=X_4=1\in \mathbb R$ we get
$$c_{n}^{(G,*),\delta}( \mathbb M_{-1}^{\gamma_1})=\frac{1}{8}R(1,1, 1, 1)=\frac{1}{8}\sum_{e_2, e_3, e_4\in \{\pm 1\}} (1+e_2e_3e_4)(1+e_2+e_3+e_4)^n=4^{n-1}.$$

\end{proof}

\subsection*{Declaration of competing interest}

The authors declare that they have no known competing financial interests or personal relationships that could have appeared to influence the work reported in this paper.


\end{document}